\documentclass[letterpaper, 10 pt, conference]{ieeeconf}  

\IEEEoverridecommandlockouts                        

\usepackage{graphics} 
\usepackage{epsfig} 
\usepackage{times} 
\usepackage{amsmath} 
\usepackage{amssymb}  
\usepackage{algorithm}
\usepackage{algpseudocode}
\usepackage{graphicx}
\usepackage{subfigure}
\usepackage{mathtools}
\usepackage{multirow}
\usepackage{xcolor}
\usepackage{hyperref}
\hypersetup{
    colorlinks=true,
    linkcolor=blue,
    filecolor=magenta,      
    urlcolor=cyan,
    pdftitle={Overleaf Example},
    pdfpagemode=FullScreen,
    }
\makeatletter
\renewenvironment{proof}[1][\relax]{\par
  \pushQED{\qed}%
  \normalfont \topsep6\p@\@plus6\p@\relax
  \trivlist
  \item[\hskip\labelsep\itshape
    \ifx#1\relax \proofname\else\proofname{} of #1\fi\@addpunct{.}]\ignorespaces
}{%
  \popQED\endtrivlist\@endpefalse
}
\makeatother

\newcommand{\III}{\mathcal{I}}
\newcommand{\JJJ}{\mathcal{J}}
\newcommand{\KKK}{\mathcal{K}}
\newcommand{\OOO}{\mathcal{O}}

\usepackage{amsthm}
\theoremstyle{plain}
\newtheorem{theorem}{Theorem}

\newtheorem{lemma}{Lemma}

\theoremstyle{definition}

\newtheorem{assumption}{Assumption}
\theoremstyle{remark}

\title{\LARGE \bf
Online Stochastic Allocation of Reusable Resources
}

\author{Xilin Zhang and Wang Chi Cheung
\thanks{Both authors are with Department of Industrial and Systems Engineering, Faculty of Engineering, National University of Singapore. (emails: {\tt\small e0408730@u.nus.edu} and {\tt\small isecwc@nus.edu.sg})}%
}

\allowdisplaybreaks
\begin{document}

\maketitle
\thispagestyle{empty}
\pagestyle{empty}

\begin{abstract}

We study a multi-objective model on the allocation of reusable resources under model uncertainty. Heterogeneous customers arrive sequentially according to a latent stochastic process, request for certain amounts of resources, and occupy them for random durations of time. The decision maker's goal is to simultaneously maximize multiple types of rewards generated by the customers, while satisfying the resource capacity constraints in each time step. We develop models and algorithms for deciding on the allocation actions. We show that when the usage duration is relatively small compared with the length of the planning horizon, our policy achieves $1-O(\epsilon)$ fraction of the optimal expected rewards, where $\epsilon$ decays to zero at a near optimal rate as the resource capacities grow. 

\end{abstract}

\section{Introduction}
In the online optimization framework, information is revealed sequentially in time. The decisions are made without knowledge of the future information, but they can depend on past observations. In this work, we study online optimization algorithms in \emph{reusable} resource allocation applications, where a resource unit is returned to the system after a period of usage duration, and can be further assigned to another customer. The decision-maker (DM) assigns limited inventories of reusable resources to sequentially arriving customers. In each time step, the DM's decision leads to a set of allocation outcomes, consisting of the amounts of rewards earned, the amounts of resources consumed and the usage durations of the assigned resources. Our model captures a diversity of real-life applications include hotel booking, rental of cars and fashion items and cloud computing services. Our problem instance incorporates the following features:
\begin{enumerate}
\item \emph{Multiple objectives.} The DM's goal is to maximize multiple types of rewards.
\item \emph{Customer heterogeneity.} The customers are associated with different customer types. 
\item \emph{Online setting.} In each time step, the arriving customer's type is drawn independently and identically from an \emph{unknown} probability distribution.
\item \emph{Reusability.} Each type of resource is endowed with a stochastic usage duration, whose probability distribution is known to the DM. However, the DM does not know the the realized usage duration of an allocation before the resource is returned.
\end{enumerate}
Features 1-3 are shared by both non-reusable and reusable resource allocation problems, while feature 4 is a distinct feature of reusable resource allocation problems. Without considering feature 4, our problem reduces to the non-reusable setting as in \cite{devanur2019near}. On feature 3, we remark that in many applications, given a customer type, its mean allocation outcome is accessible by machine learning (ML) approaches in a data-efficient manner (\cite{chen2022statistical,han2020neural,aouad2022representing}). In many existing resource allocation research (\cite{levi2010provably,chen2017revenue,lei2020real,baek2022bifurcating,besbes2021static}), the mean allocation outcomes are assumed to be prior-knowledge acquired through ML models. However, in applications where customer types are represented as high-dimensional feature vectors, the number of types can be exponential in the dimension of the feature vectors or even unbounded. Such a curse of dimensionality hinders the estimation on the proportion of each customer type. Therefore, we treat the probability distribution of each customer type as the unknown object. We further remark on feature 4 that our usage duration is defined in the same manner as \cite{lei2020real} and \cite{rusmevichientong2020dynamic}, which are recent works on offline reusable resource allocation problems. We highlight that the probability distribution of each usage duration can be arbitrarily defined, which means our result does not depend on specific structures of certain usage distributions, such as the exponential distribution.

Traditional resource allocation problems \cite{adelman2007dynamic, alaei2012online} concern the allocation of non-reusable resources. Online algorithms for allocating non-reusable resources have been extensively studied in \cite{agrawal2014fast, AgrawalWY14, balseiro2020dual, devanur2019near, FeldmanHKMS10, yu2017online, yuan2018online}. These algorithms involve adaptive weighing processes that balance the trade-off between the rewards earned and the resources consumed. Most of their analysis largely depend on the monotonically-decreasing inventories. However, in our reusable model, the effect of an allocation may be different for each future time step, contingent on whether the allocated resources are returned, causing fluctuating resource consumption amount across consecutive time steps.

To our knowledge, this is the first paper to address reusable resource allocation problems in an online stochastic setting and demonstrate a near-optimal performance guarantee. Some studies focus on assortment planning problems in adversarial settings (\cite{feng2022near, gong2019online, goyal2020asymptotically}), and achieve non-trivial competitive ratios. Offline pricing and assortment planning problems have been studied in \cite{lei2020real,rusmevichientong2020dynamic,feng2022near}, where near-optimality is achieved in the form of approximation ratios under full model certainty. The main contribution of our paper can be summarized as follows. 
\begin{itemize}
\item \emph{Model generality.} We propose a general reusable resource allocation model which allows for various decision settings (such as admission control, matching, pricing and assortment planning), multiple objectives (such as revenue, market share and service level), and large numbers of customer types or allocation types (the algorithm's performance is independent of these sizes).
\item \emph{Near-optimal algorithm performance.} We develop an adaptive weighing algorithm that trades-off among not only the resources occupied and the rewards earned, but also the usage durations. In the regime where each usage duration is short compared with the length of the planning horizon, our algorithm achieves matching near-optimal performance as the online non-reusable resource setting (\cite{devanur2019near}), as well as the the state-of-art offline reusable setting (\cite{feng2022near}).
\end{itemize}
The remainder of paper is organized as follows. In Section \ref{sec:model}, we present our model and highlight some of its applications. An online algorithm and the corresponding performance analysis are proposed in Section \ref{sec:alg}. In Section \ref{sec:num}, numerical experiments are provided.

\section{Model}\label{sec:model}
\textbf{Notation. }The reward types and the resource types are respectively indexed by two finite sets $\III_r$ and $\III_c$. A generic reward type or resource type is denoted as $i$. For each $i\in \III_c$, the DM has $c_i\in \mathbb{R}_{>0}$ units of resource $i$ for allocation. Each customer is associated with a customer type $j\in \JJJ$, which reflects the customer's features.  
We denote the set of possible allocation actions as $\KKK$, and each element $k\in \KKK$ as an action. The action set can model a broad range of decisions, which is elaborate in the end of Section \ref{sec:model}. 

The DM allocates the resources in $T$ discrete time steps. In time step $t \in \{1,\ldots,T\}$, at most one customer arrives. We denote the customer type of the arrival at time $t$ as $j(t)$. In particular, we designate the type $j_{\textsf{null}} \in \JJJ$ to represent the case of no arrival. We assume that $j(1), \ldots, j(T)$ are independently and identically distributed (i.i.d.) random variables over $\JJJ$. We denote $p_j = \Pr(j(1) = j)$, and $\mathbf{p} = \{p_j\}_{j\in \JJJ}$. 

When a customer (denote his type as $j$) arrives, the DM chooses an action $k\in \KKK$. The choice leads to an array of stochastic outcomes, consisting of the amount of rewards earned $W_{jk} = (W_{ijk})_{i \in \III_r}$, the amount of resources occupied $A_{jk} = (A_{ijk})_{i\in \III_c}$, and the usage durations $\{D_i\}_{i\in \III_c}$. For the no arrival customer type $j_{\textsf{null}}$, we stipulate that $\Pr(W_{i', j_{\textsf{null}}, k} = A_{i, j_{\textsf{null}}, k} = 0) = 1$ for all $i' \in \III_r, i\in \III_c, k\in \KKK$, since there should be no reward earned and no resource occupied in the case of no arrival. To ensure feasibility in our resource constrained model, we assume that there exists a null action $k_{\textsf{null}}\in \KKK$ that satisfies $\Pr(W_{i', j, k_{\textsf{null}}} = A_{i, j, k_{\textsf{null}}} = 0) = 1$ for all $i' \in \III_r, i\in \III_c, j\in \JJJ$. Selecting the null action is equivalent to rejecting a customer. 

For each $j, k$, the stochastic outcomes follow the joint distribution $\OOO_{jk}$, namely $(W_{jk}, A_{jk})\sim \OOO_{jk}$. We allow $W_{jk}, A_{jk}$ to be arbitrarily correlated. For each $i\in \III_c$, the random usage duration $D_i$ is independent of $W_{jk}, A_{jk}$. This assumption is also made in related works on offline reusable resource allocation, such as \cite{lei2020real} and \cite{rusmevichientong2020dynamic}, since the usage duration reflects more of a customer's intrinsic needs on each resource. We assume that $W_{ijk}\in [0,w_{\max}]$ for each $i \in \III_r, j\in \JJJ, k\in \KKK$, $A_{ijk}\in [0,a_{\max}]$ almost surely for each $i \in \III_c, j\in \JJJ, k\in \KKK$, and $D_i \in \{0,1, \ldots, d_{\max}\}$ almost surely for each $i \in \III_c$. Additionally, we denote $w_{ijk} = \mathbb{E}[W_{ijk}]$, $a_{ijk} = \mathbb{E}[A_{ijk}]$, and $d_i = \mathbb{E}[ D_i]$.

\textbf{Model uncertainty and dynamics.} We assume that the DM knows the horizon length $T$, the values of $w_{\max}$, $a_{\max}$, $d_{\max}$, as well as $\Pr( D_i \geq t)$ for every $i \in \III_c, t \in \{1, \ldots, T\}$. However, the DM does not know the probability distribution  $\mathbf{p}$ over customer types. At each time step $t\in \{1, \ldots, T\}$, the DM observes the type $j(t)\sim \mathbf{p}$ of the arriving customer, and the mean outcomes $\{(w_{j(t), k}, a_{j(t), k})\}_{k\in \KKK}$ specific to the type $j(t)$. Then, the DM chooses an action $k(t)\in \KKK$, and observes the stochastic outcomes of rewards $\{W_{i, j(t), k(t)}(t)\}_{i \in \III_r}$ and resources $\{A_{i, j(t), k(t)}(t)\}_{i \in \III_c}$ at time $t$. Our model uncertainty scenario included the case when the DM knows the mean outcomes $a_{ijk},w_{ijk}$ in advance. For example, the DM could have estimates on $a_{ijk}, w_{ijk}, \Pr( D_i \geq t)$ by constructing supervised learning models \cite{chen2022statistical, han2020neural, aouad2022representing} on a pool of customer demand data.  

\textbf{An integer programming formulation.} We let binary decision variables $X^\pi_k(t)$ be the DM's decision under a non-anticipatory algorithm $\pi$, where $X^\pi_k(t) =1$ if action $k$ is taken at time $t$, and $X^\pi_k(t) =0$ otherwise. Under a non-anticiaptory algorithm, $\{X^\pi_k(t)\}_{k\in \KKK}$ depends on historical observations $\{j(s)\}^{t}_{s=1} \cup \{W_{i, j(s), k(s)}(s)\}_{i\in \III_r, 1\leq s\leq t-1} \cup \{A_{i, j(s), k(s)}(s)\}_{i\in \III_c, 1\leq s\leq t-1}$. The DM aims to maximize $\mathbb{E}[\min_{i \in \III_r} \sum^T_{t=1} \sum_{k\in \KKK} W_{i,j(t),k}(t) X^\pi_{k}(t)]$, which achieves the simultaneous maximization of all the reward types by ensuring max-min fairness. Here we maximize the rewards to keep in line with the resource allocation literature instead of minimize the regret as in the classical online convex optimization literature, but we remark that they are essentially eqivalent. For each $i \in \III_c$ and $t \in \{1,\ldots,T\}$, we require that the resource constraint
\begin{equation}
\sum^t_{\tau = 1} \sum_{k\in \KKK} \mathbf{1}( D_i(\tau) \geq t - \tau + 1) A_{i,j(\tau),k}(\tau) X^\pi_{k}(\tau) \leq c_i \label{oeq:res_cons}
\end{equation} 
holds with certainty. The left hand side in (\ref{oeq:res_cons}) represents the amount of type $i$ resources occupied at time step $t$. Our goal can be formulated as the following binary integer program
\begin{subequations}
\begin{alignat}{2}
&\text{(IP-C) } \max\limits_{\text{non-anticipatory }\pi}  ~\mathbb{E}[\hat{\lambda}] \nonumber\\
&\text{s.t.}  ~\sum^T_{t=1} \sum_{k\in \KKK} W_{i,j(t),k}(t) X^\pi_{k}(t)\geq T\hat{\lambda}     \quad \forall i\in \III_r   \nonumber\\
&\quad \sum^t_{\tau = 1} \sum_{k\in \KKK} \mathbf{1}( D_i(\tau) \geq t - \tau + 1) A_{i,j(\tau),k}(\tau) X^\pi_{k}(\tau) \leq c_i \nonumber\\
&\quad \quad \quad \quad \quad \quad \quad \quad \quad \quad \quad \quad \quad \forall i\in \III_c,~ t\in \{1, \ldots T\} \nonumber\\
&\quad \sum_{k\in \KKK} X^\pi_k(t) = 1 \quad \forall ~ t\in [T] \nonumber\\
&\quad X^\pi_k(t) \in \{0,1\}      \quad \forall k\in \KKK,~ t\in [T] \nonumber.
\end{alignat}
\end{subequations}
We remark that the term $\mathbf{1}( D_i(\tau) \geq t - \tau + 1)$ induces non-stationarity in resource consumption, since even when a DM selects the same action $k$ at $\tau_1< \tau_2$, their amounts of resource consumption $\mathbf{1}( D_i(\tau_1) \geq t - \tau_1 + 1)A_{i,j(\tau_1),k}(\tau_1), \mathbf{1}( D_i(\tau_2) \geq t - \tau_2 + 1) A_{i,j(\tau_2),k}(\tau_2)$ at time $t$ are differently distributed. Existing works on non-reusable resources crucially hinges on model stationarity, which does not hold true in our setting. 

\textbf{A tractable benchmark.} The goal of constructing a non-anticipatory algorithm that achieves the optimal value of (IP-C) is analytically intractable due to the curse of dimensionality. The intractability motivates us to consider an alternative linear program (LP), dubbed (LP-E), where the realization of the customer arrivals, their usage duration and outcomes exactly follow the expectation:
\begin{align}
&\text{(LP-E) }\max  ~\lambda & \nonumber\\
&\text{s.t.} \sum^T_{t=1} \sum_{j \in \JJJ} \sum_{k\in \KKK} p_j w_{ijk}y_{jk}(t)\geq T \lambda     \quad \forall i\in \III_r   \nonumber\\
& \quad \sum^t_{\tau = 1}  \sum_{j\in \JJJ}\sum_{k\in \KKK} p_j \Pr( D_i \geq t - \tau + 1) a_{ijk} y_{jk}(\tau) \leq c_i  \nonumber \\
&\quad \quad \quad \quad \quad \quad \quad \quad \quad \quad \quad \quad \quad \forall i\in \III_c,~ t\in \{1, \ldots T\} \nonumber\\
&\quad \sum_{k\in \KKK}y_{jk}(t)\leq 1 \quad \forall j\in \JJJ,~ t\in \{1, \ldots, T\} \nonumber\\
&\quad y_{jk}(t)\geq 0      \quad \forall j\in \JJJ, ~k\in \KKK,~ t\in [T] \nonumber.
\end{align}
Define the optimal objective value of (LP-E) to be $\lambda^*$, and let the optimal objective of (IP-C) be $\hat{\lambda}^*$. The following lemma shows that $\lambda^*$ is a tractable upper bound for the expected reward of any online algorithms.
\begin{lemma}\label{olem:2.1}
$\lambda^* \geq \mathbb{E}[\hat{\lambda}^*]$.
\end{lemma}
For the algorithm design, we further introduce a ``steady state'' benchmark, assuming the decision variables are invariant across time:
\begin{subequations}
\begin{alignat}{2}
\text{ (LP-SS): }
&\max\limits_{x_{jk}}  ~\Tilde{\lambda} & \nonumber\\
\text{s.t.}  &\sum_{j \in \JJJ} \sum_{k\in \KKK} p_j w_{ijk}x_{jk}\geq \Tilde{\lambda}     &\quad &\forall i\in \III_r   \nonumber\\
&\sum_{j\in \JJJ}\sum_{k\in \KKK} p_j a_{ijk} d_i x_{jk} \leq c_i       &\quad & \forall i\in \III_c \nonumber\\
&\sum_{k\in \KKK} x_{jk} \leq 1      &\quad &\forall j\in \JJJ \nonumber\\
&x_{jk}\geq 0      &\quad &\forall j\in \JJJ, ~k\in \KKK \nonumber.
\end{alignat}
\end{subequations}
We denote an optimal solution of (LP-SS) as $x^*_{jk}$, and the optimal value of (LP-SS) as $\Tilde{\lambda}^*$. We further define $$\gamma=\min\left\{\min_{i\in \III_c}\left\{\frac{c_i}{a_{\max}}\right\}, \frac{T\Tilde{\lambda}^*}{w_{\max}}\right\}.$$ 
\begin{assumption}\label{asp:dur}
There exists $\delta\in (0,1)$ and $\bar{d}(\delta) \leq T$  such that $\sum^{\infty}_{t=\bar{d}(\delta)+1} \Pr(D_i \geq t) \leq \delta$, $\forall i,j,k$.  
\end{assumption}
This assumption indicates that our algorithm does not apply to large $D_i$, say for non-reusable resources where $D_i=T$ with certainty. In the next lemma, we show that $\Tilde{\lambda}^*$ is close to $\lambda^*$. 
\begin{lemma}\label{olem:3.2}
$ \left(1 - \frac{\delta}{\gamma}\right) \left(T \lambda^* - \bar{d}(\delta) w_{\max}\right) \leq T \Tilde{\lambda}^* \leq T \lambda^*$.
\end{lemma}
We remark that under a wide range of usage durations, we can use (LP-SS) as a benchmark to gauge the performance of our algorithm. For instance, for light tailed $D_i$ (for example, there exists $u>0$ such that $\lim_{t \rightarrow \infty} \Pr(D_i \geq t) t^u = 0$), we can take $\delta= \epsilon / T$, $\bar{d}(\delta) = d_{max} \log (d_{\max} T /\epsilon )$. If $D_i$ has bounded support, i.e. $D_i \in [0,d_{\text{UB}}]$ almost surely, we can take $\delta=0$ and let $\bar{d}(\delta)=d_{\text{UB}}$.

\textbf{Applications.} Before proceeding to our algorithm development, we highlight the generality of our model by discussing some of its applications, where the reward type set $\III_r$, the customer type set $\JJJ$ and the action set $\KKK$ can be chosen to model a variety of decisions. 
\begin{itemize}
\item \textbf{Admission control.} In this setting, the DM is to either admit or reject each arriving customer \cite{levi2010provably}. Real life examples include patient inflow control in an emergency department or an ICU. The admission control setting can be modeled by letting action set $\KKK = \{\text{accept}, \text{reject}\}$. The reward of a resource is fixed at $r_i$ for $i \in \III_c$. Upon taking an action $k$ for a type $j$ customer, an array of stochastic demands $\{A_{ijk}\}_{i \in \III_c}$ is generated. Our model captures different reward settings. We list some of the examples: for simultaneously maximizing the revenue/social welfare for each type of resource, define $\III_r = \III_c$ and $W_{ijk} = r_i A_{ijk}$. For maximizing the total revenue/social welfare of all resources, let $\III_r=\{1\}$ be a singleton, and define $W_{1jk}=\sum_{i \in \III_c} r_i A_{ijk}$. For maximizing the service level of each resource, we define $\III_r = \III_c$ and $W_{ijk} = A_{ijk}$. We remark that multiple kinds of rewards can be modeled simultaneously.
\item \textbf{Assortment Planning.} In assortment planning problems, one unit of resource $i$ is associated with a fixed price $r_i$. The DM influences the customers' demands through offering different assortments of resources. Real life assortment planning examples with reusable resources include renting of fashion items and vehicles. Contingent upon the arrival of a customer, say of type $j$, the DM decides the assortment $k \in \KKK$ to display, where $\KKK$ is a collection of subsets of $\III_c$ (\cite{rusmevichientong2020dynamic,feng2019linear}). Let $q_{ijk}$ denote the probability for customer type $j$ to choose product $i$ in assortment $k$. In the revenue management literature, the probability $q_{ijk}$ is modeled by a random utility choice model. The assortment planning problem (simultaneously maximizing revenue of each resource) can be incorporated in our model by setting $\III_r = \III_c$, setting $A_{ijk}$ to be the Bernoulli random variable with mean $q_{ijk}$, and setting $W_{ijk}=r_i A_{ijk}$.
\end{itemize}

\section{Online algorithm and performance analysis}\label{sec:alg}

\textbf{Main results.} In this section, we propose a multi-stage adaptive weighing algorithm (dubbed Algorithm $A$ as in ``Adaptive'', and displayed in Algorithm \ref{alg:one}). We first provide our main results. We assume there exists $\epsilon$ satisfying $\gamma =  \Omega( \log ( |\III| T/\epsilon) / \epsilon^2)$, where $\III = \III_c \cup \III_r$. Let $$W_i(t) = \sum_{k \in \KKK} W_{i,j(t),k}(t) X^{A}_{k}(t)$$ denote the type $i \in \III_r$ reward achieved by Algorithm $A$, our main result is shown in the following theorem.

\begin{theorem}\label{othm:main}
Let $\epsilon>0$ be an arbitrary constant satisfying $\gamma= \Omega( \log ( |\III| T / \epsilon) / \epsilon^2 )$. Without knowing $\boldsymbol{p}$, \begin{align}
\sum^{T}_{t=1} W_i(t) \geq \left(1 - \frac{\delta}{\gamma}\right) T \lambda^* (1-O(\epsilon)) - \bar{d}(\delta) \Tilde{O}(\epsilon) \nonumber
\end{align}
for every $i\in \III_r$ with probability at least $1-\epsilon (1+\epsilon)^{\delta}$. 
\end{theorem}
We remark that if $D_i$ has bounded support, i.e. $D_i \in [0,d_{\text{UB}}]$ almost surely for all $i \in \III_c$, then the above reward can be simplified as $\sum^{T}_{t=1} W_i(t) \geq T \lambda^* (1-O(\epsilon)) - d_{\text{UB}} \Tilde{O}(\epsilon)$ for every $i\in \III_r$ with probability at least $1-\epsilon$. In the case when $\delta \leq \log(|\III|T)/\epsilon$ and $\bar{d}(\delta) = o(T)$, 
our algorithm achieves a reward at least $T \lambda^* (1-O(\epsilon))$. This result nearly matches the \cite{devanur2019near} in studying an online non-reusable resource allocation problem, as well as \cite{feng2022near} in the state-of-art work on the offline assortment planning with reusable resources.

The assumption $\gamma = \Omega( \log ( |\III| T / \epsilon) / \epsilon^2 )$ means that the amount of resource $c_i$ for each $i$ is sufficiently large, and the planning horizon $T$ is sufficiently long. Crucially, the performance guarantee in Theorem \ref{othm:main} does not deteriorate even when $|\JJJ|$ or $|\KKK |$ grows. Consequently, our Algorithm $A$ is applicable to complex scenarios when $|\JJJ| > T$. 

\textbf{High-level description and comparison against \cite{devanur2019near}.} Our Algorithm $A$ extends \cite{devanur2019near}'s idea of adaptive weighing from the non-reusable setting to the reusable setting. \cite{devanur2019near} proposes a multi-stage adaptive weighing algorithm to trade-off between the rewards and the resources. For each reward and resource constraint of a fluid LP (corresponding to our (LP-E)), a penalty weight is defined. The weight on each constraint progressively gets larger as the reward generated or the resource consumed gets closer to their total capacity (the capacity of each reward is approximated). \cite{devanur2019near} does not apply to the reusable setting, as their penalty weights and algorithm design depend on the monotonic-decreasing resources.

Somewhat surprisingly, we can still utilize their adaptive weighing idea. we approximate (LP-E) with a knapsack-constrained (LP-SS), and define penalty weights that incorporates the usage duration as well as the resource consumption. In a series of Lemmas that eventually leads to Theorem \ref{othm:main}, we show that our algorithm effectively capitalizes the reusability of the resources to maximize the total rewards. We overcome the technical difficulty in non-monotonic and time-correlated resource levels, by achieving near-optimal performance. Nevertheless, we remark that the closeness of (LP-E) and (LP-SS) builds upon Assumption \ref{asp:dur}, and hence our algorithm is not applicable to the non-reusable setting.

\textbf{A multistage online algorithm.} In Algorithm \ref{alg:one}, we divide the time horizon into $l$ stages $\{-1, 0,1,\ldots,l-1\}$ where $l$ satisfies $\epsilon 2^l=1$ for some $\epsilon \in [\bar{d}(\delta)/T, 1/2]$. Stage $-1$ consists of $t^{(-1)}=\epsilon T$ time periods. This stage is solely for exploration on the latent $\{p_j\}_{j\in \JJJ}$, and the first $\epsilon T$ customers are served with random actions. Stage $r \in \{0,\ldots,l-1\}$ consists of $t^{(r)}=\epsilon T 2^r$ time periods. The assumption $\epsilon \in [\bar{d}(\delta)/T, 1/2]$ ensures that $l \geq 1$ (there is at least 1 stage) and $\epsilon T \geq \bar{d}(\delta)$ (each stage consists of at least $\bar{d}(\delta)$ time periods). We denote $j^{(r)}(s)$ as the type of the customer who arrives at the $s$-th time step in stage $r$ (where $s \in \{1,\ldots,t^{(r)}\}$), meaning that $j^{(r)}(s) = j(t^{(r)}+s)$.  

In each stage $r \geq 0$, Algorithm $A$ consists of two steps. In \textbf{Step 1}, we estimate the optimum of (LP-SS). In \textbf{Step 2}, we define ``penalty weights'' on constraints of (LP-SS), and choose the action that balances between each constraint. 

\textbf{Step 1: }\emph{Estimate the value of $\Tilde{\lambda}^*$ (Line \ref{oalg:step_1} of Algorithm \ref{alg:one}).} We first derive $\mu^{(r)*}$, which is the optimal objective value of the linear program $\text{(LP-RSS)}^{(r)}:$ 
\begin{align}
 &~\max\limits_{x^{(r)}_{jk}}  ~\mu^{(r)} & \nonumber\\
\text{s.t.}  &\sum_{j \in \JJJ} \sum_{k\in \KKK} \hat{p}^{(r)}_{j} w_{ijk} x^{(r)}_{jk} \geq \mu_r     &\quad &\forall i\in \III_r   \nonumber\\
&\sum_{j\in \JJJ}\sum_{k\in \KKK} \hat{p}^{(r)}_{j} a_{ijk} d_i x^{(r)}_{jk} \leq c_i       &\quad & \forall i\in \III_c \nonumber\\
&\sum_{k\in \KKK} x^{(r)}_{jk}\leq 1      &\quad &\forall j\in \JJJ \nonumber\\
&x^{(r)}_{jk} \geq 0      &\quad &\forall j\in \JJJ, ~k\in \KKK \nonumber,
\end{align}
where $\hat{p}^{(r)}_{j}=\frac{1}{t^{(r-1)}} \sum^{t^{(r-1)}}_{t=1} \mathbf{1}(j^{(r-1)}(t)=j)$, denoting the empirical customer distribution based on customer arrivals in stage $r-1$. $\text{(LP-RSS)}^{(r)}$ is a sample average approximation (SAA) of (LP-SS). It is worth mentioning that both (LP-SS) and $\text{(LP-RSS)}^{(r)}$ are highly tractable, even in assortment planning application when $|\KKK|$ is exponential in $|\III_c|$ (\cite{bront2009column}). In addition, the knapsack-type constraints in $\text{(LP-RSS)}^{(r)}$ allows us to apply the adaptive weighing in Step 2 in a similar manner to the non-reusable setting. Given that $\mu^{(r)*}$ is easily obtained in Step 1, we define $\lambda^{(r)}$ in the following lemma, and show it is a progressively more accurate estimate of $\Tilde{\lambda}^*$ as $r$ grows. 

\begin{lemma}\label{olem:3.3}
Define $\epsilon^{(r)}_{x} = \sqrt{\frac{4 T \log\frac{2 |\III|}{\eta}}{t^{(r)} \gamma}}$ for $r \in \{-1,0,1,\ldots,l-1\}$. For any $\eta \in (0,1)$, with probability at least $1-2 \eta$,
\begin{equation}
\Tilde{\lambda}^*(1-3\epsilon^{(r-1)}_{x}) \leq \lambda^{(r)} \leq \Tilde{\lambda}^* \nonumber
\end{equation}
where $\lambda^{(r)} = \frac{\mu^{(r)*}}{1+\epsilon^{(r-1)}_{x}}$.
\end{lemma}

\textbf{Step 2:} \emph{Run an online algorithm given $\lambda^{(r)}$ (Line \ref{oalg:step_2_start} - Line \ref{oalg:step_2_end} of Algorithm \ref{alg:one}).} With slight abuse of notation, we write $A_{i,j(t^{(r)}+s),k}(t^{(r)}+s)$ as $A^{(r)}_{i,j^{(r)}(s),k}(s)$, $W_{i,j(t^{(r)}+s),k}(t^{(r)}+s)$ as $W^{(r)}_{i,j^{(r)}(s),k}(s)$, and $D_{i,j(t^{(r)}+s),k}(t^{(r)}+s)$ as $D^{(r)}_{i}(s)$. In addition, we denote $X^{(r)}_{j(t^{(r)}+s),k}(t^{(r)}+s)$ as  $X^{(r)}_{j^{(r)}(s),k}(s)$. Define, at time $t$ in stage $r$, $Y^{(r)}_{i \tau t} = \sum_{k \in \KKK} \mathbf{1}(D^{(r)}_{i}(\tau) \geq t - \tau +1) A^{(r)}_{i,j^{(r)}(\tau),k}(\tau) X^{(r)}_{j^{(r)}(\tau),k}(\tau)$ and $Z^{(r)}_{it} = \sum_{k \in \KKK} W^{(r)}_{i,j^{(r)}(t),k}(t) X^{(r)}_{j^{(r)}(t),k}(t)$ respectively as resource $i$ consumed by customer $\tau$, and reward $i$ earned. 

At the $s$-th time step of stage $r$, after observing the customer type $j^{(r)}(s)$ we take action $k^{(r)}(s)$ according to Line \ref{oalg:k}. The parameter
$\phi^{(r)}_{i,s,t}$ represents a ``penalty weight" for the resource constraint $i\in \III_c$ in (LP-SS). If the allocation decisions during  $1, \ldots, s-1$ in stage $r$ leads to a high amount of resource $i$ occupation at time $t$, the penalty $\phi^{(r)}_{i,s,t}$ would also be high. Similarly, a lower amount of accrued reward type $i \in \III_r$ during $1, \ldots, s-1$ in stage $r$ leads to a higher value of the weight $\psi^{(r)}_{i, s}$. Both weights quantify the DM's emphasis on resources and rewards. 

\begin{algorithm}[]
\caption{Online Algorithm $A$}\label{alg:one}
\textbf{Input:} the number of time periods $T$, the capacities for each resource $c_i$, the values of $\gamma$ and $\epsilon \in \left[d_i/T, 1/2\right]$. \\
\textbf{Output:} actions to take $k^{(r)}(t)$, for $r=0,\ldots,l-1, t=1,\dots,t^{(r)}$.
\begin{algorithmic}[1]
\State Set $l=\log\left(1/\epsilon\right)$.
Initialize $t^{(-1)}=\epsilon T$.
\For{$r = 0, \ldots, l-1$}
\State Compute $\lambda^{(r)}$ by solving $\text{(LP-RSS)}^{(r)}$. \label{oalg:step_1}
\State \label{oalg:step_2_start} Set $$\epsilon^{(r-1)}_{x} = \sqrt{\frac{4 T \log\frac{2 |\III|}{\eta}}{t^{(r-1)} \gamma}},~ \epsilon^{(r)}_{z}=\sqrt{\frac{2w_{\max} (1+\epsilon) \log\frac{2|\III|l}{\eta}}{t^{(r)} \lambda^{(r)}}}.$$
\State Set 
\[
\phi^{(r)}_{i,1,t}= 
\begin{cases}
\frac{\epsilon \gamma}{c_i (1+\epsilon)^{\gamma-\delta}},  ~t=1 \\
\frac{\epsilon \gamma}{c_i (1+\epsilon)^{\gamma-\delta}} \prod^{t}_{\tau=2} \left(1+\epsilon {\frac{\gamma \Pr\left(D_i \geq t-\tau+1\right)}{d_i (1+\epsilon)}}\right),  \\ \quad \quad \quad \quad \quad \quad \quad \quad \quad \quad \quad \quad t=2,\ldots,t^{(r)}
\end{cases}
\]
for each $i\in \III_c$, and
\begin{align}
\psi^{(r)}_{i,1}=& -\frac{\epsilon^{(r)}_{z} \prod^{t^{(r)}}_{\tau=2} \left(1-\epsilon^{(r)}_{z} \frac{\lambda^{(r)}}{w_{\max} (1+\epsilon)}\right)}{w_{\max} \left(1-\epsilon^{(r)}_{z}\right)^{\frac{\left(1-\epsilon^{(r)}_{z}\right) t^{(r)} \lambda^{(r)}}{w_{\max}}}}.  \nonumber
\end{align}
for each $i\in \III_r$.
\For{$s = 1, \ldots, t^{(r)}$}
\State \label{oalg:k} Observe customer type $j^{(r)}(s)$, take action:
\begin{align}
& k^{(r)}(s) \in \arg \min \limits_{k \in \KKK} \nonumber \\
& \left\{\sum^{t^{(r)}}_{t=s} \sum_{i \in \III_c} a_{i,j^{(r)}(s),k} \Pr\left(D_i \geq t-s+1\right) \phi^{(r)}_{i,s,t} \right. \nonumber \\
& \qquad + \left. \sum_{i \in \III_r} w_{i,j^{(r)}(s),k} \psi^{(r)}_{i,s}\right\}. \nonumber
\end{align}
\State Set $Y^{(r)}_{ist}=a_{i,j^{(r)}(s),k^{(r)}(s)} \Pr\left(D_i \geq t-s+1\right)$ for each $i\in \III_s$ and $t\in \{s, \ldots, t^{(r)}\}$, and $Z^{(r)}_{is}=w_{i,j^{(r)}(s),k^{(r)}(s)}$ for each $i\in \III_r$.
\State\label{oalg:weight} Set for each $i\in \III_c$
\begin{align}
&\phi^{(r)}_{i,s+1,t}=\frac{\phi^{(r)}_{i,s,t} (1+\epsilon)^{\frac{\gamma}{c_i} Y^{(r)}_{ist}}}{\left(1+\epsilon \frac{\gamma \Pr\left(D_i \geq t-s\right)}{d_i (1+\epsilon)}\right)}, ~t=s+1,\ldots,t^{(r)}, \nonumber 
\end{align}
and for each $i\in \III_r$,
\begin{align}
&\psi^{(r)}_{i,s+1}= \frac{\psi^{(r)}_{i,s} \left(1-\epsilon^{(r)}_{z}\right)^{\frac{1}{w_{\max}} Z^{(r)}_{is}}}{\left(1-\epsilon^{(r)}_{z} \frac{\lambda^{(r)}}{w_{\max}(1+\epsilon)}\right)}.  \nonumber
\end{align}
\EndFor \label{oalg:step_2_end}
\EndFor
\end{algorithmic}
\end{algorithm}

\textbf{Performance guarantee of Algorithm $A$.} Our analysis involves bounding the total probability of violating each constraint in (LP-C) in all time steps. Each violation leads to unserved customers and lost rewards. We show that under Algorithm $A$, all constraints of (LP-C) are satisfied with high probability. To understand the choice of $k^{(r)}(s+1)$ in Line \ref{oalg:k} of Algorithm \ref{alg:one}, we introduce an auxiliary offline static algorithm (dubbed Algorithm $S$ as in ``Static''). Algorithm $S$ requires knowing $x^* = \{x^*_{jk}\}_{jk}$, an optimal solution to (LP-SS), and a tuning parameter $\epsilon\in (0, 1)$ for preserving capacities in anticipation of any constraint violation. At a time $t$, if a customer of type $j$ arrives, the DM selects action $k\in \KKK\setminus k_{\text{null}}$ with probability $\frac{x_{jk}^*}{1+\epsilon}$, and selects the null action $k_{\text{null}}$ with probability $\frac{\epsilon}{1+\epsilon} + \frac{x_{j,k_{\text{null}}}^*}{1+\epsilon}$. We denote $X^S_{j(t^{(r)}+s),k}(t^{(r)}+s)$ as $\tilde{X}^{(r)}_{j^{(r)}(s),k}(s)$. Define $\tilde{Y}^{(r)}_{i \tau t} = \sum_{k \in \KKK} \mathbf{1}(D^{(r)}_{i}(\tau) \geq t - \tau +1) A^{(r)}_{i,j^{(r)}(\tau),k}(\tau) \Tilde{X}^{(r)}_{j^{(r)}(\tau)k,r}(\tau)$ and $\Tilde{Z}^{(r)}_{it} = \sum_{k \in \KKK} W^{(r)}_{i,j^{(r)}(t),k}(t) \Tilde{X}^{(r)}_{j^{(r)}(t),k}(t)$ where $\Pr(\Tilde{X}^{(r)}_{j^{(r)}(\tau),k}(\tau)=1 )=\frac{x^*_{jk}}{1+\epsilon}$ for each $\tau$ in stage $r$. A performance guarantee of Algorithm $S$ is provided in the following lemma.

\begin{lemma}\label{olem:3.4}
Let $\eta=\epsilon/(5l)$, Algorithm $S$ achieves a total reward of at least
\begin{subequations}
\begin{alignat}{2}
&\sum^{l-1}_{r=0} \sum^{t^{(r)}}_{t=1} \sum_{k \in \KKK} W^{(r)}_{i,j^{(r)}(t),k}(t) \Tilde{X}^{(r)}_{j^{(r)}(t),k}(t) \nonumber \\ 
\geq & T \Tilde{\lambda}^* (1-O(\epsilon)) - \bar{d}(\delta) w_{\max} O\left(\epsilon + \log \frac{1}{\epsilon}\right) \nonumber
\end{alignat}
\end{subequations}
for every $i\in \III_r$ with probability at least $1-\epsilon (1+\epsilon)^{\delta}$. 
\end{lemma}
For analysis sake, we consider a hybrid Algorithm $A_s S_{t^{(r)}-s}$ in each stage $r$. For Algorithm $A_s S_{t^{(r)}-s}$, the DM makes allocation decisions based on Algorithm $A$ in time step $\{1,\ldots, s\}$, and based on Algorithm $S$ in time step $\{s+1,\ldots,t^{(r)}\}$. We show that $A_{s+1} S_{t^{(r)}-s-1}$ outperforms $A_s S_{t^{(r)}-s}$, which inductively leads to the conclusion that the performance of the online adaptive Algorithm $A$ is no worse than the offline static Algorithm $S$ (see Lemma \ref{olem:3.4}). This induction technique is introduced in \cite{devanur2019near} on the non-reusable setting. We need more refined techniques to disentangle the time-correlation between the resources constraints at each time step, and finally prove the following lemma.
\begin{lemma}\label{olem:3.5}
Algorithm $A$ achieves a total reward of at least 
\begin{align}
& \sum^{l-1}_{r=0} \sum^{t^{(r)}}_{t=1} \sum_{k \in \KKK} W_{ij(t)k,r}(t) \Tilde{X}^{(r)}_{j(t)k,r}(t) \nonumber \\
\geq & T \Tilde{\lambda}^* (1-O(\epsilon)) - \bar{d}(\delta) w_{\max} O\left(\epsilon + \log \frac{1}{\epsilon}\right) \nonumber
\end{align}
for every $i\in \III_r$ with probability at least $1-\epsilon (1+\epsilon)^{\delta}$.
\end{lemma}
Putting Lemmas \ref{olem:3.2} and \ref{olem:3.5} together, we have Theorem \ref{othm:main}. 

\section{Numerical experiments}\label{sec:num}
We consider an assortment planning problem. One unit of resource $i$ is associated with a fixed price $r_i$. Contingent upon the arrival of a customer, say of type $j$, the DM decides the assortment $k \in \KKK$ to display, where $\KKK$ is a collection of subsets of $\III_c$. Let $q_{ijk}$ denote the probability for customer type $j$ to choose product $i$ in assortment $k$. Therefore, the assortment planning problem (simultaneously maximizing the revenue of each resource) can be incorporated in our model by setting $\III_r = \III_c$, setting $A_{ijk}$ to be the Bernoulli random variable with mean $q_{ijk}$, and setting $W_{ijk}=r_i A_{ijk}$. In our test, the probability $q_{ijk}$ is modeled by the multinomial logit (MNL) choice model. Each resource $i\in \III_c$ is associated with a feature vector $\boldsymbol{f}_i \in \mathbb{R}^m$, and each customer type $j\in \JJJ$ is associated with a set of feature vectors $\{\boldsymbol{b}_{ij} \}_{i \in \III_c}$, where $\boldsymbol{b}_{ij}  \in \mathbb{R}^m$ for each $i\in \III_c$. The feature vector $\boldsymbol{f}_i$ could involve the fixed price $r_i$, and $q_{ijk} = \frac{\exp(\boldsymbol{b}_{ij}^\top \boldsymbol{f}_i)}{1+\sum_{\ell \in k} \exp(\boldsymbol{b}_{\ell j}^\top \boldsymbol{f}_\ell)}$ if $i\in k$, and $q_{ijk} = 0$ if $i\not \in k$. In complement, the probability of no purchase is  $\frac{1}{1+\sum_{\ell \in k} \exp(\boldsymbol{b}_{\ell j}^\top \boldsymbol{f}_\ell)}.$ Notice that the size of the action set $\KKK$ scales exponentially with the number of products. Nevertheless, for the MNL models, $k^{(r)}(t)$ can be computed efficiently by solving a simple LP whose computational time is polynomial in $|\III_c|$ \cite{davis2013assortment}. 

We consider a synthetic data-set with $14$ types of resources indexed by $\III_c=\{1,2,\dots,14\}$, and $1000$ types of customers indexed by $\JJJ=\{1,2,\ldots,1000\}$. We allow offering any assortment of fewer than 5 products, and hence the assortment set is of size $\sum^5_{i=1} C^{14}_i=3472$. We let $\boldsymbol{p}$ follow a discrete distribution with a support of $|\JJJ|$. For any $n \geq 1$, we set $T=1000n$, $c_i=20n$ and $D_i$ follow a randomly generated probability distribution with a bounded support of $[1,200n]$. Theoretically, the regret of both Algorithms $A$ and $S$ should grow sublinearly against the scale $n$, roughly at the rate of $\Tilde{O}(\sqrt{n})$. For each $T$ value, we run $10$ simulations using a column generation approach and take the average as well as the standard deviation.

\begin{table}[h]
\centering
\resizebox{0.5\textwidth}{!}{%
\begin{tabular}{|c|c|c|cccc|cc|}
\hline
\multirow{3}{*}{$T$} &
  \multirow{3}{*}{$\epsilon$} &
  \multirow{3}{*}{UB} &
  \multicolumn{4}{c|}{Total Revenue} &
  \multicolumn{2}{c|}{$\%$ Gap from UB} \\
 &
   &
   &
  \multicolumn{2}{c}{Algo $S$} &
  \multicolumn{2}{c|}{Algo $A$} &
  \multirow{2}{*}{Algo $S$} &
  \multirow{2}{*}{Algo $A$} \\
     &       &         & Mean    & Std                           & Mean    & Std      &         &         \\ \hline
1000 & 0.3   & 644.90  & 507.06  & \multicolumn{1}{c|}{42.1689}  & 331.03  & 2.531798 & 21.37\% & 48.67\% \\
2000 & 0.22  & 1289.80 & 1097.47 & \multicolumn{1}{c|}{50.88593} & 887.28  & 7.681146 & 14.91\% & 31.21\% \\
3000 & 0.185 & 1934.70 & 1702.97 & \multicolumn{1}{c|}{75.61902} & 1419.54 & 8.537564 & 11.98\% & 26.63\% \\
4000 & 0.162 & 2579.60 & 2327.51 & \multicolumn{1}{c|}{86.79659} & 1966.83 & 6.663332 & 9.77\%  & 23.75\% \\
5000 & 0.148 & 3224.50 & 2936.00 & \multicolumn{1}{c|}{98.77243} & 2522.00 & 10.44462 & 8.95\%  & 21.79\% \\
6000 & 0.136 & 3869.40 & 3557.05 & \multicolumn{1}{c|}{74.11313} & 3086.28 & 20.91315 & 8.07\%  & 20.24\% \\
7000 & 0.127 & 4514.30 & 4200.24 & \multicolumn{1}{c|}{67.91298} & 3647.95 & 9.508417 & 6.96\%  & 19.19\% \\
8000 & 0.12  & 5159.20 & 4804.85 & \multicolumn{1}{c|}{68.66264} & 4217.17 & 14.06983 & 6.87\%  & 18.26\% \\ \hline
\end{tabular}%
}
\caption{Results with $14$ resource types and $1000$ customer types.}
\label{tab:my-table}
\end{table}
In Table \ref{tab:my-table}, the third column of upper bounds are the optimal values of (LP-SS). The offline Algorithm $S$ performs better than the online Algorithm $A$. Algorithm $A$ achieves rewards within $1-2\epsilon$ fraction of the upper bounds.

\newpage
\onecolumn
\section*{Appendix}
\subsection{Proof of Lemma \ref{olem:2.1}}
\begin{proof}
Take the expectation of the resource constraints of the (LP-C) $\sum^t_{\tau = 1} \sum_{k\in \KKK} \mathbf{1}( D_i(\tau) \geq t - \tau + 1) A_{i,j(\tau),k}(\tau) X^{\pi}_{k}(\tau)$ over $X^{\pi}_k(\tau)$, $ D_i(\tau)$, $A_{i,j(\tau),k}(\tau)$ and $j(\tau)$ for $\tau=1,\ldots,t$:
\begin{subequations}
\begin{alignat}{2}
& \mathbb{E}\left[\sum^t_{\tau = 1} \sum_{k\in \KKK} \mathbf{1}\left( D_i\left(\tau\right) \geq t - \tau + 1\right) A_{i,j\left(\tau\right),k}\left(\tau\right) X^{\pi}_{k}\left(\tau\right)\right] \nonumber \\
=& \sum^t_{\tau = 1} \sum_{k \in \KKK} \sum_{j\left(1\right),\ldots,j\left(t\right)} \mathbb{E}\left[\mathbf{1}\left( D_i\left(\tau\right) \geq t - \tau + 1\right)A_{i,j\left(\tau\right),k}\left(\tau\right) X^{\pi}_{k}\left(\tau\right)|j\left(1\right),\ldots,j\left(t\right)\right] \Pr\left(j\left(1\right),\ldots,j\left(t\right)\right) \nonumber \\
=& \sum^t_{\tau = 1} \sum_{k \in \KKK} \sum_{j\left(1\right),\ldots,j\left(t\right)} \mathbb{E}\left[\mathbf{1}\left( D_i\left(\tau\right) \geq t - \tau + 1\right)|j\left(1\right),\ldots,j\left(t\right)\right] \cdot \mathbb{E}\left[A_{i,j\left(\tau\right),k}\left(\tau\right)|j\left(1\right),\ldots,j\left(t\right)\right] \nonumber \\ 
&\quad \quad \quad \quad \quad \quad ~\cdot \mathbb{E}\left[X^{\pi}_{k}\left(\tau\right)|j\left(1\right),\ldots,j\left(t\right)\right] \cdot \Pr\left(j\left(1\right),\ldots,j\left(t\right)\right) \nonumber \\
=& \sum^t_{\tau = 1} \sum_{k \in \KKK} \sum_{j\left(1\right),\ldots,j\left(t\right)} \mathbb{E}\left[\mathbf{1}\left( D_i\left(\tau\right) \geq t - \tau + 1\right)|j\left(\tau\right)\right] \cdot \mathbb{E}\left[A_{i,j\left(\tau\right),k}\left(\tau\right)|j\left(\tau\right)\right] \nonumber \\ 
&\quad \quad \quad \quad \quad \quad ~\cdot \mathbb{E}\left[X^{\pi}_{k}\left(\tau\right)|j\left(1\right),\ldots,j\left(t\right)\right] \cdot \Pr\left(j\left(1\right),\ldots,j\left(t\right)\right) \nonumber \\
=& \sum^t_{\tau = 1} \sum_{k \in \KKK} \sum_{j\left(\tau\right)} p_{j\left(\tau\right)} a_{i,j\left(\tau\right),k} \Pr\left( D_i \geq t-\tau+1\right) \nonumber \\
&\quad \quad \quad \quad ~ \cdot \sum_{j\left(1\right),\ldots,j\left(\tau-1\right),j\left(\tau+1\right),\ldots,j\left(t\right)} \mathbb{E}\left[X^{\pi}_{k}\left(\tau\right)|j\left(1\right),\ldots,j\left(t\right)\right] \Pr\left(j\left(1\right),\ldots,j\left(\tau-1\right),j\left(\tau+1\right),\ldots,j\left(t\right)\right) \nonumber \\
=& \sum^t_{\tau = 1} \sum_{k \in \KKK} \sum_{j \in \JJJ} p_{j} a_{ijk} \Pr\left( D_i \geq t-\tau+1\right) \nonumber \\
&\quad \quad \quad \quad ~ \cdot \sum_{j\left(1\right),\ldots,j\left(\tau-1\right),j\left(\tau+1\right),\ldots,j\left(t\right)} \mathbb{E}\left[X^{\pi}_{k}\left(\tau\right)|j\left(1\right),\ldots,j\left(t\right)\right] \Pr\left(j\left(1\right),\ldots,j\left(\tau-1\right),j\left(\tau+1\right),\ldots,j\left(t\right)\right) \nonumber \\
\leq & c_i \nonumber
\end{alignat}
\end{subequations}
The first equation follows from the summation of probabilities. The second equation holds because $ D_i(\tau)$ is independent of $A_{i,j(\tau),k}(\tau)$, and the decision is made before the realization of $A_{i,j(\tau),k}(\tau)$ and $ D_i(\tau)$. The third inequality is valid since $ D_i(\tau)$ and $A_{i,j(\tau),k}(\tau)$ only depend on $j(\tau)$. The forth inequality follows by separately consider $j(\tau)$ and customers at other time periods. The fifth inequality stands since $j(1),\dots,j(T)$ are i.i.d. across time, and we can simply denote $j(\tau)$ as $j$. The last inequality holds since the clairvoyant's LP is feasible for all $j(1), \ldots, j(T)$ paths. Letting $x_{jk}(\tau)=\sum_{j(1),\ldots,j(\tau-1),j(\tau+1),\ldots,j(t)} \mathbb{E}[X^{\pi}_{k}(\tau)|j(1),\ldots,j(t)] \Pr(j(1),\ldots,j(\tau-1),j(\tau+1),\ldots,j(t))$, we have a feasible solution for the online DM's expected LP. For the reward constraints, $\mathbb{E}[\sum^T_{t=1} \sum_{k\in \KKK} W_{ij(t)k}(t) X^{\pi}_{k}(t)]=\mathbb{E}[T \hat{\lambda}^*]$. Since the solution $x_{jk}(\tau)$ achieves an objective of $\mathbb{E}[T \hat{\lambda}^*]$ and it is a feasible solution of the DM's expected LP, we have $\lambda^* \geq \mathbb{E}[\hat{\lambda}^*]$.
\end{proof}

\subsection{Proof of Lemma \ref{olem:3.2}}
\begin{proof}
For each $i$ we have $d_i=\sum^{\infty}_{t=1} \Pr( D_i \geq t)$. Hence, the capacity constraints of (LP-SS) can be expressed as:
\begin{subequations}
\begin{alignat}{2}
c_i & \geq \sum_{j\in \JJJ}\sum_{k\in \KKK} p_j a_{ijk} d_i x^*_{jk} \nonumber \\
& = \sum_{j\in \JJJ} \sum_{k\in \KKK} p_j a_{ijk} \sum^{\infty}_{t=1} \Pr( D_i \geq t) x^*_{jk} \nonumber \\
& \geq \sum_{j\in \JJJ} \sum_{k\in \KKK} p_j a_{ijk} \sum^{t}_{\tau=1} \Pr( D_i \geq t-\tau+1) x^*_{jk}
\quad \forall i\in \III_c,~ t\in \{1, \ldots, T\}. \nonumber
\end{alignat}
\end{subequations}
From the above, it is evident that the optimal solution $x^*_{jk}$ of (LP-SS) is feasible for the (LP-E) by letting $y_{jk}(t)=x^*_{jk}$ for all $t \in \{1,\ldots,T\}$, and therefore $\Tilde{\lambda}^* \leq \lambda^*$. In contrast, the optimal solution $y^*_{jk}(t)$ of (LP-E) may not be feasible for (LP-SS).

We consider a truncated version (LP-TE) of the DM's expected LP (LP-E). For each resource constraint:
\begin{subequations}
\begin{alignat}{2}
\text{(LP-TE)}:~ \max\limits &~\lambda & \nonumber\\
\text{s.t.}  &\sum^T_{t=1} \sum_{j \in \JJJ} \sum_{k\in \KKK} p_j w_{ijk}y_{jk}(t)\geq T \lambda     &\quad &\forall i\in \III_r   \nonumber\\
&\sum^t_{\tau = \max(t-\bar{d}(\delta),1)}  \sum_{j\in \JJJ}\sum_{k\in \KKK} p_j \Pr(D_i \geq t - \tau + 1) a_{ijk} y_{jk}(\tau) \leq c_i       &\quad & \forall i\in \III_c,~ t\in \{1, \ldots T\} \nonumber\\
&\sum_{k\in \KKK}y_{jk}(t)\leq 1      &\quad &\forall j\in \JJJ,~ t\in \{1, \ldots, T\} \nonumber\\
&y_{jk}(t)\geq 0      &\quad &\forall j\in \JJJ, ~k\in \KKK,~ t\in [T] \nonumber.
\end{alignat}
\end{subequations}
Denote for (LP-TE), the optimal objective value as $\lambda'$ and the optimal solution as $y'_{jk}(t)$ for $t\in \{1, \ldots, T\}$. Since the feasible region (LP-TE) contains the feasible region of (LP-E) and both LPs have the same objective, we have $\lambda' \geq \lambda^*$. Note that $\{y'_{jk}(t)\}_{j,k,t}$ may not be feasible for (LP-E). Despite that, we can show that $y_{jk}(t)=\left(1-\frac{\delta}{\gamma}\right) y'_{jk}(t)$ is feasible for (LP-E):
\begin{subequations}
\begin{alignat}{2}
&\sum^t_{\tau = 1}  \sum_{j\in \JJJ}\sum_{k\in \KKK} p_j \Pr\left(D_{i} \geq t - \tau + 1\right) a_{ijk} \left(1-\frac{\delta}{\gamma}\right) y'_{jk}\left(\tau\right) \nonumber \\
\leq & \sum^t_{\tau = \max\left(t-\bar{d}\left(\delta\right),1\right)}  \sum_{j\in \JJJ}\sum_{k\in \KKK} p_j \Pr\left(D_{i} \geq t - \tau + 1\right) a_{ijk} \left(1-\frac{\delta}{\gamma}\right) y'_{jk}\left(\tau\right) \nonumber \\
& + \sum^{\max\left(0,t-\bar{d}\left(\delta\right)-1\right)}_{\tau = 1}  \sum_{j\in \JJJ}\sum_{k\in \KKK} p_j \Pr\left(D_{i} \geq t - \tau + 1\right) a_{ijk} \left(1-\frac{\delta}{\gamma}\right) y'_{jk}\left(\tau\right) \nonumber \\
\leq & \left(1-\frac{\delta}{\gamma}\right) c_i + a_{max} \delta \nonumber \\
= & \left(1-\frac{\delta}{\gamma} + \frac{\delta a_{max}}{c_i}\right) c_i \nonumber \\
\leq & c_i. \nonumber
\end{alignat}
\end{subequations}
By letting $y_{jk}\left(t\right)=\left(1-\frac{\delta}{\gamma}\right) y'_{jk}\left(t\right)$, the objective value is at least $\sum^T_{t=1} \sum_{j \in \JJJ} \sum_{k\in \KKK} p_j w_{ijk} \left(1-\frac{\delta}{\gamma}\right) y'_{jk}\left(t\right) \geq \left(1-\frac{\delta}{\gamma}\right) T\lambda'$. Hence $\lambda^* \geq \left(1-\frac{\delta}{\gamma}\right) \lambda'$. We formulate a truncated version of the steady-state LP (LP-SS) in a similar fashion:
\begin{subequations}
\begin{alignat}{2}
\text{(LP-TSS) }: ~\max\limits_{x_{jk}}  &~\Tilde{\lambda} & \nonumber\\
\text{s.t.}  &\sum_{j \in \JJJ} \sum_{k\in \KKK} p_j w_{ijk}x_{jk}\geq \Tilde{\lambda}     &\quad &\forall i\in \III_r   \nonumber\\
&\sum_{j\in \JJJ}\sum_{k\in \KKK} \sum^{\bar{d}(\delta)}_{\tau = 1} p_j a_{ijk} \Pr(D_{i} \geq \tau) x_{jk} \leq c_i       &\quad & \forall i\in \III_c \nonumber\\
&\sum_{k\in \KKK} x_{jk} \leq 1      &\quad &\forall j\in \JJJ \nonumber\\
&x_{jk}\geq 0      &\quad &\forall j\in \JJJ, ~k\in \KKK \nonumber.
\end{alignat}
\end{subequations}
Denote for the (LP-TSS), the optimal objective value as $\Tilde{\lambda}'$ and the optimal solution as $x'_{jk}$. Similar to (LP-TE), we have $\Tilde{\lambda}' \geq \Tilde{\lambda}^*$ and $\Tilde{\lambda}^* \geq \left(1-\frac{\delta}{\gamma}\right) \Tilde{\lambda}'$. It's also obvious that $\lambda' \geq \Tilde{\lambda}'$ since we can let $x'_{jk}(\tau)=x'_{jk}$ and get a feasible solution for the (LP-TE). Next we show that $\lambda'$ is not so far away from $\Tilde{\lambda}'$.

The dual of the truncated DM's expected LP is:
\begin{subequations}
\begin{alignat}{2}
&\text{(LP-TE-D)}: ~\min\limits_{\boldsymbol{\alpha}^{(1)}, \boldsymbol{\beta}^{(1)}, \boldsymbol{\rho}^{(1)}} ~ \sum^T_{t=1} \left(\sum_{j \in \JJJ} p_j \beta^{(1)}_{jt} + \sum_{i \in \III_c} c_i \alpha^{(1)}_{it}\right) & \nonumber\\
\text{s.t.}  &~\beta^{(1)}_{jt} + \sum_{i \in \III_c} a_{ijk} \sum^{\min(t+\bar{d}(\delta),T)}_{\tau=t} \Pr(D_{i} \geq \tau-t+1)  \alpha^{(1)}_{i \tau} - \sum_{i \in \III_r} w_{ijk} \rho^{(1)}_i \geq 0   &\quad &\forall j \in \JJJ, ~k \in \KKK, \nonumber \\
&&&~t \in \{1, \ldots, T\}   \nonumber\\
&\sum_{i \in \III_r} \rho^{(1)}_i \geq 1   & \nonumber\\
&\rho^{(1)}_i \geq 0  & \quad &\forall i \in \III_r \nonumber\\
&\alpha^{(1)}_{it} \geq 0   &\quad &\forall i \in \III_c, ~t \in \{1, \ldots, T\} \nonumber\\
&\beta^{(1)}_{jt} \geq 0      &\quad &\forall j\in \JJJ, ~t \in \{1, \ldots, T\} \nonumber.
\end{alignat}
\end{subequations}

The dual of truncated steady-state LP is:
\begin{subequations}
\begin{alignat}{2}
\text{(LP-TSS-D)}: &~\min\limits_{\boldsymbol{\alpha}^{(2)}, \boldsymbol{\beta}^{(2)}, \boldsymbol{\rho}^{(2)}}  ~ \sum_{j \in \JJJ} p_j \beta^{(2)}_j + \sum_{i \in \III_c} c_i \alpha^{(2)}_i & \nonumber \\
\text{s.t.}  &~\beta^{(2)}_j + \sum_{i \in \III_c} a_{ijk} \sum^{\bar{d}(\delta)}_{\tau = 1} \Pr(D_{i} \geq \tau) \alpha^{(2)}_i - \sum_{i \in \III_r} w_{ijk} \rho^{(2)}_i \geq 0   &\quad &\forall j \in \JJJ, ~k \in \KKK   \nonumber\\
&\sum_{i \in \III_r} \rho^{(2)}_i \geq 1   & \nonumber\\
&\rho^{(2)}_i \geq 0  &\quad &\forall i \in \III_r \nonumber\\
&\alpha^{(2)}_i \geq 0   &\quad &\forall i \in \III_c \nonumber\\
&\beta^{(2)}_j \geq 0      &\quad &\forall j\in \JJJ \nonumber.
\end{alignat}
\end{subequations}
Obviously, the optimal objective value of (LP-TE-D) $T \lambda'$ is larger than or equal to $T$ times the optimal objective value of (LP-TSS-D) $T \Tilde{\lambda}'$. Let $\boldsymbol{\alpha}^{(2)*}, \boldsymbol{\beta}^{(2)*}, \boldsymbol{\rho}^{(2)*}$ denote the optimal solution of (LP-TSS-D). Since these are minimization problems, if for (LP-TE-D) we can let $\alpha^{(1)}_{it}=\alpha^{(2)*}_i$, $\beta^{(1)}_{jt}=\beta^{(2)*}_j$ for all $t$, and $\rho^{(1)}_i=\rho^{(2)*}_i$ and still have a feasible LP, then $\lambda'=\Tilde{\lambda}'$. Consider the first set of constraints in (LP-TE-D). For $t \in \{1, \ldots, T-\bar{d}(\delta)\}$, by letting $\alpha^{(1)}_{it}=\alpha^{(2)*}_i$, $\beta^{(1)}_{jt}=\beta^{(2)*}_j$ and $\rho^{(1)}_i=\rho^{(2)*}_i$, we have:
\begin{subequations}
\begin{alignat}{2}
& \beta^{(1)}_{jt} + \sum_{i \in \III} a_{ijk} \sum^{\min(t+\bar{d}(\delta),T)}_{\tau=t} \Pr(D_{i} \geq \tau-t+1)  \alpha^{(1)}_{i \tau} - \sum_{i \in \III} w_{ijk} \rho^{(1)}_i \nonumber \\
= & \beta^{(2)*}_j + \sum_{i \in \III} a_{ijk} \sum^{t+\bar{d}(\delta)}_{\tau = t} \Pr(D_{i} \geq \tau-t+1) \alpha^{(2)*}_i - \sum_{i \in \III} w_{ijk} \rho^{(2)*}_i \nonumber \\
= & \beta^{(2)*}_j + \sum_{i \in \III} a_{ijk} \sum^{\bar{d}(\delta)}_{\tau = 1} \Pr(D_{i} \geq \tau) \alpha^{(2)*}_i - \sum_{i \in \III} w_{ijk} \rho^{(2)*}_i \nonumber \\
\geq & 0. \nonumber
\end{alignat}
\end{subequations}

However, for $t \in \{T-\bar{d}(\delta)+1, \ldots, T\}$, the constraints may be violated:
\begin{subequations}
\begin{alignat}{2}
& \beta^{(1)}_{jt} + \sum_{i \in \III} a_{ijk} \sum^{\min(t+\bar{d}(\delta),T)}_{\tau=t} \Pr(D_{i} \geq \tau-t+1)  \alpha^{(1)}_{i \tau} - \sum_{i \in \III} w_{ijk} \rho^{(1)}_i \nonumber \\
= & \beta^{(2)*}_j + \sum_{i \in \III} a_{ijk} \sum^{T}_{\tau = t} \Pr(D_{i} \geq \tau-t+1) \alpha^{(2)*}_i - \sum_{i \in \III} w_{ijk} \rho^{(2)*}_i \nonumber \\
= & \beta^{(2)*}_j + \sum_{i \in \III} a_{ijk} \sum^{T-t}_{\tau = 1} \Pr(D_{i} \geq \tau) \alpha^{(2)*}_i - \sum_{i \in \III} w_{ijk} \rho^{(2)*}_i \nonumber
\end{alignat}
\end{subequations}
where $T-t \leq \bar{d}(\delta)$. Notice that since (LP-TE-D) is a minimization problem, $\sum_{i \in \III}\rho^{(1)*}_i$ is never strictly larger than $1$. Hence to make the above expression larger than or equal to $0$, it's enough to let $\beta^{(1)}_{jt}=w_{\max}$. In all, by letting $\alpha^{(1)}_{it}=\alpha^{(2)*}_i$ for $t \in \{1, \ldots, T\}$, $\beta^{(1)}_{jt}=\beta^{(2)*}_j$ for $t \in \{1, \ldots, T-\bar{d}(\delta)\}$, $\beta^{(1)}_{jt}=w_{\max}$ for $t \in \{T-\bar{d}(\delta)+1, \ldots, T\}$ and $\rho^{(1)}_i=\rho^{(2)*}_i$, we have a feasible solution for the (LP-TE-D). Then $T \lambda' - T \Tilde{\lambda}' \leq \sum^T_{t=T-\bar{d}(\delta)+1} \sum_{j \in \JJJ} p_j \beta^{(1)}_{jt} \leq \bar{d}(\delta) w_{\max}$. Concluding the above analysis, we have:
\begin{subequations}
\begin{alignat}{2}
& T \lambda' \geq T \lambda^* \geq T \left(1 - \frac{\delta}{\gamma}\right) \lambda' \nonumber \\
& T \Tilde{\lambda}' \geq T \Tilde{\lambda}^* \geq T \left(1 - \frac{\delta}{\gamma}\right) \Tilde{\lambda}' \nonumber \\
& T \lambda' \leq T \Tilde{\lambda}' + \bar{d}(\delta) w_{\max}. \nonumber
\end{alignat}
\end{subequations}
Putting these together, we have:
\begin{subequations}
\begin{alignat}{2}
T \Tilde{\lambda}^* \geq \left(1 - \frac{\delta}{\gamma}\right) (T \lambda^* - \bar{d}(\delta) w_{\max}). \nonumber
\end{alignat}
\end{subequations}
\end{proof}

\subsection{Proof of Lemma \ref{olem:3.3}}
Before proceeding to the proof, we introduce the multiplicative Chernoff bounds.
\begin{lemma}[Multiplicative Chernoff bounds]
Let $X=\sum_i X_i$, where $X_i \in [0,B]$ are independent random variables, and let $\mathbb{E}[X]=\mu$.

(a) For all $\epsilon>0$,
\begin{subequations}
\begin{alignat}{2}
\Pr(X<\mu(1-\epsilon)) < \exp\left(\frac{-\epsilon^2 \mu}{2B}\right). \nonumber
\end{alignat}
\end{subequations}
Therefore, for all $\delta>0$, with probability at least $1-\delta$,
\begin{subequations}
\begin{alignat}{2}
X-\mu \geq -\sqrt{2\mu B \ln(1/\delta)}. \nonumber
\end{alignat}
\end{subequations}

(b) For $\epsilon \in [0,2e-1]$,
\begin{subequations}
\begin{alignat}{2}
\Pr(X>\mu(1+\epsilon)) < \exp \left(\frac{-\epsilon^2 \mu}{4B}\right). \nonumber
\end{alignat}
\end{subequations}
Hence, for all $\delta > \exp \left(\frac{-(2e-1)^2 \mu}{4B}\right)$, with probability at least $1-\delta$,
\begin{subequations}
\begin{alignat}{2}
X-\mu \leq \sqrt{4\mu B \ln(1/\delta)}. \nonumber
\end{alignat}
\end{subequations}
For $\epsilon>2e-1$,
\begin{subequations}
\begin{alignat}{2}
\Pr(X>\mu(1+\epsilon)) < 2^{-(1+\epsilon) \mu/B}. \nonumber
\end{alignat}
\end{subequations}
\end{lemma}

\begin{proof}[Lemma \ref{olem:3.3}]
Since $\lambda^{(r)} = \frac{\mu^{(r)*}}{1+\epsilon^{(r-1)}_{x}}$, showing $\Tilde{\lambda}^*(1-3\epsilon^{(r-1)}_{x}) \leq \lambda^{(r)} \leq \Tilde{\lambda}^*$ is equivalent to showing $\Tilde{\lambda}^*(1-2\epsilon^{(r-1)}_{x}) \leq \mu^{(r)*} \leq \Tilde{\lambda}^* (1+\epsilon^{(r-1)}_{x})$. We first show that the lower bound $\Tilde{\lambda}^* (1-2 \epsilon^{(r-1)}_{x}) \leq \mu^{(r)*}$ holds with probability $1-\eta$. It's obvious that the expected instance of $\text{(LP-RSS)}^{(r)}$ has the same optimal solution and objective function as (LP-SS). In $\text{(LP-RSS)}^{(r)}$, letting $x^{(r)}_{jk}=\frac{x^*_{jk}}{1+\epsilon^{(r-1)}_{x}}$, the probability of violating each resource constraint is:
\begin{subequations}
\begin{alignat}{2}
&\Pr\left(\sum_{j\in \JJJ}\sum_{k\in \KKK} \hat{p}^{(r)}_{j} a_{ijk} d_i \frac{x^*_{jk}}{1+\epsilon^{(r-1)}_{x}} \geq \frac{c_i}{1+\epsilon^{(r-1)}_{x}} \left(1+\epsilon^{(r-1)}_{x}\right)\right) \nonumber\\
\leq &\Pr\left(\sum_{j\in \JJJ}\sum_{k\in \KKK} \hat{p}^{(r)}_{j} d_i \frac{x^*_{jk}}{1+\epsilon^{(r-1)}_{x}} \geq \frac{\gamma}{1+\epsilon^{(r-1)}_{x}} \left(1+\epsilon^{(r-1)}_{x}\right)\right) \nonumber\\
\leq &\Pr\left(\sum^{t^{(r)}}_{t=t^{(r-1)}+1} \sum_{j\in \JJJ} \sum_{k\in \KKK} \frac{\mathbf{1}\left(j\left(t\right)=j\right)}{t^{(r-1)}} \frac{x^*_{jk}}{1+\epsilon^{(r-1)}_{x}} \geq \frac{\gamma}{d_i \left(1+\epsilon^{(r-1)}_{x}\right)} \left(1+\epsilon^{(r-1)}_{x}\right)\right) \nonumber \\
\leq & \exp\left(\frac{- \left(\epsilon^{(r-1)}_{x}\right)^2 t^{(r-1)} \gamma}{4 d_i \left(1+\epsilon^{(r-1)}_{x}\right)}\right) \nonumber \\
\leq & \frac{\eta}{2|\III|} \nonumber.
\end{alignat}
\end{subequations}
The first and second inequalities stand by the definition of $\gamma$ and $\hat{p}^{(r)}_{j}$ respectively. For the third inequality, since $j(t)$'s are independent, $\sum_{j\in \JJJ} \sum_{k\in \KKK} \frac{\mathbf{1}\left(j\left(t\right)=j\right)}{t^{(r-1)}} \frac{x^*_{jk}}{1+\epsilon^{(r-1)}_{x}} \leq \frac{1}{t^{(r-1)}}$ and $$\mathbb{E}\left[\sum^{t^{(r)}}_{t=t^{(r-1)}+1} \sum_{j\in \JJJ} \sum_{k\in \KKK} \frac{\mathbf{1}\left(j\left(t\right)=j\right)}{t^{(r-1)}} \frac{x^*_{jk}}{1+\epsilon^{(r-1)}_{x}}\right] \leq \frac{\gamma}{d_i \left(1+\epsilon^{(r-1)}_{x}\right)},$$ we apply the multiplicative Chernoff bounds to derive the result. The fourth inequality follows if $$\gamma \geq \frac{d_i (1+\epsilon^{(r-1)}_{x}) \log \frac{2|\III|}{\eta}}{t^{(r-1)} \left(\epsilon^{(r-1)}_{x}\right)^2} = \frac{d_i (1+\epsilon^{(r-1)}_{x})}{4T} \gamma.$$ This is indeed the case since we assume the usage duration $d_i ~\forall i \in \III_c$ is small compared with $T$. Taking a union bound over all $i \in \III_c$, the probability of violating any resource constraints is at most $\sum_{i \in \III_c} \frac{\eta}{2|\III|} \leq \frac{\eta}{2}$. 

A similar process holds for bounding the probability of violating each reward constraint:
\begin{subequations}
\begin{alignat}{2}
&\Pr\left(\sum_{j\in \JJJ}\sum_{k\in \KKK} \hat{p}^{(r)}_{j} w_{ijk} \frac{x^*_{jk}}{1+\epsilon^{(r-1)}_{x}} \leq \frac{\Tilde{\lambda}^*}{1+\epsilon^{(r-1)}_{x}} \left(1-\epsilon^{(r-1)}_{x}\right)\right) \nonumber\\
\leq &\Pr\left(\sum^{t^{(r)}}_{t=t^{(r-1)}+1} \sum_{j\in \JJJ} \sum_{k\in \KKK} \frac{\mathbf{1}\left(j\left(t\right)=j\right)}{t^{(r-1)}} \frac{w_{ijk} \gamma}{\Tilde{\lambda}^*} \frac{x^*_{jk}}{1+\epsilon^{(r-1)}_{x}} \leq \frac{\gamma}{1+\epsilon^{(r-1)}_{x}} \left(1-\epsilon^{(r-1)}_{x}\right)\right) \nonumber \\
\leq & \exp\left(\frac{- \left(\epsilon^{(r-1)}_{x}\right)^2 t^{(r-1)} \gamma}{2T \left(1+\epsilon^{(r-1)}_{x}\right)}\right) \nonumber \\
\leq & \frac{\eta}{2|\III|} \nonumber.
\end{alignat}
\end{subequations}
Combining the union bounds for all resource constraints and reward constraints, with with probability at least $1-\eta$, the stage $r$ LP $\text{(LP-RSS)}^{(r)}$ achieves an objective value of at least $\Tilde{\lambda}^* \frac{1-\epsilon^{(r-1)}_{x}}{1+\epsilon^{(r-1)}_{x}} \geq \Tilde{\lambda}^* (1-2\epsilon^{(r-1)}_{x})$.

Next we show that the upper bound $\mu^{(r)*} \leq \Tilde{\lambda}^*(1+ \epsilon^{(r-1)}_{x})$ holds with probability $1-\eta$. Note that the dual of (LP-SS) is:
\begin{subequations}
\begin{alignat}{2}
\text{(LP-SS-D) }: \min\limits_{\boldsymbol{\alpha}, \boldsymbol{\beta}, \boldsymbol{\rho}}  &~ \sum_{j \in \JJJ} p_j \beta_j + \sum_{i \in \III_c} c_i \alpha_i & \nonumber\\
\text{s.t.}  &~\beta_j + \sum_{i \in \III_c} a_{ijk} d_i \alpha_i - \sum_{i \in \III_r} w_{ijk} \rho_i \geq 0   &\quad &\forall j \in \JJJ, ~k \in \KKK   \nonumber\\
&\sum_{i \in \III_r} \rho_i \geq 1   & \nonumber\\
&\rho_i \geq 0   &\quad &\forall i \in \III_r \nonumber\\ 
&\alpha_i \geq 0   &\quad &\forall i \in \III_c \nonumber\\
&\beta_j \geq 0      &\quad &\forall j\in \JJJ \nonumber.
\end{alignat}
\end{subequations}
The dual of $\text{(LP-RSS)}^{(r)}$ is:
\begin{subequations}
\begin{alignat}{2}
\text{(LP-RSS-D)}^{(r)}: ~\min\limits_{\boldsymbol{\alpha}^{(r)}, \boldsymbol{\beta}^{(r)}, \boldsymbol{\rho}^{(r)}}  &~ \sum_{j \in \JJJ} \hat{p}^{(r)}_{j} \beta^{(r)}_{j} + \sum_{i \in \III_c} c_i \alpha^{(r)}_{i} & \nonumber\\
\text{s.t.}  &~\beta^{(r)}_{j} + \sum_{i \in \III_c} a_{ijk} d_i \alpha^{(r)}_{i} - \sum_{i \in \III_r} w_{ijk} \rho^{(r)}_{i} \geq 0   &\quad &\forall j \in \JJJ, ~k \in \KKK   \nonumber\\
&\sum_{i \in \III_r} \rho^{(r)}_{i} \geq 1   & \nonumber\\
&\rho^{(r)}_{i}  \geq 0   &\quad &\forall i \in \III_r \nonumber\\
&\alpha^{(r)}_{i} \geq 0   &\quad &\forall i \in \III_c \nonumber\\
&\beta^{(r)}_{j} \geq 0      &\quad &\forall j\in \JJJ \nonumber.
\end{alignat}
\end{subequations}

Since the domain of the constraints are exactly the same for two dual LPs, the optimal solution of (LP-SS-D) is feasible for $\text{(LP-RSS-D)}^{(r)}$. Suppose the optimal solutions to (LP-SS-D) is $\alpha^*_i$, $\beta^*_j$ and $\rho^*_i$. The optimal objective value is then $\Tilde{\lambda}^* =\sum_{j \in \JJJ} p_j \beta^*_j + \sum_{i \in \III_c} c_i \alpha^*_i$. Letting $\alpha^{(r)}_{i}=\alpha^*_i$, $\beta^{(r)}_{j}=\beta^*_j$ and $\rho^{(r)}_{i}=\rho^*_i$, since $\text{(LP-RSS-D)}^{(r)}$ is a minimization problem, we have an upper bound for optimal objective value $\mu^{(r)*}$, i.e. $\mu^{(r)*} \leq \sum_{j \in \JJJ} \hat{p}^{(r)}_{j} \beta^*_j + \sum_{i \in \III_c} c_i \alpha^*_i$. If $\sum_{i \in \III_r} \rho^*_i$ is strictly larger than $1$, we can lower the values of $\rho^*_i$ and still get a feasible solution for (LP-SS-D). In this case the objective is not optimal. Therefore, $\sum_{i \in \III_r} \rho^*_i=1$. It can be seen that $\beta^*_j + \sum_{i \in \III_c} a_{ijk} d_i \alpha^*_i \leq w_{\max}$, since otherwise we can lower the value of $\beta^*_j$ and still have a feasible solution. Hence if $\sqrt{\frac{4 w_{\max} \log (1/\eta)}{t^{(r-1)}(\sum_{j \in \JJJ} p_j \beta^*_j)}} \in [0,2e-1]$, then with probability at most $1-\eta$ we have
\begin{subequations}
\begin{alignat}{2}
&\sum_{j \in \JJJ} \hat{p}^{(r)}_{j} \beta^*_j - \sum_{j \in \JJJ} p_j \beta^*_j \nonumber \\
= &\sum^{t^{(r)}}_{t=t^{(r-1)}+1} \sum_{j \in \JJJ} \frac{\mathbf{1}(j(t)=j)}{t^{(r-1)}} \beta^*_j - \sum_{j \in \JJJ} p_j \beta^*_j \nonumber \\
\leq &\sqrt{\frac{4(\sum_{j \in \JJJ} p_j \beta^*_j) w_{\max} \log(1/\eta)}{t^{(r-1)}}}. \nonumber
\end{alignat}
\end{subequations}
The inequality is derived by applying the Chernoff bounds to $\sum_{j \in \JJJ} \frac{\mathbf{1}(j(t)=j)}{t^{(r-1)}} \beta^*_j$. We have $\frac{w_{\max}}{t^{(r-1)}}$ under the square root since $\beta^*_j \leq w_{\max}$ and therefore $\sum_{j \in \JJJ} \frac{\mathbf{1}(j(t)=j)}{t^{(r-1)}} \beta^*_j \leq \frac{w_{\max}}{t^{(r-1)}}$. By the assumption that $\gamma = \min\limits_{i \in \III_c} \left(\frac{c_i}{a_{\max}}, \frac{T\Tilde{\lambda}^*}{w_{\max}}\right) \leq \frac{T\Tilde{\lambda}^*}{w_{\max}}$ and definition of $\epsilon^{(r-1)}_{x}$, we have:
\begin{subequations}
\begin{alignat}{2}
&\sum_{j \in \JJJ} p_j \beta^*_j \leq \Tilde{\lambda}^* \nonumber \\
\Rightarrow &\frac{(\sum_{j \in \JJJ} p_j \beta^*_j) w_{\max}}{\left(\Tilde{\lambda}^{*}\right)^2} \leq \frac{w_{\max}}{\Tilde{\lambda}^*} \nonumber \\
\Rightarrow &\frac{(\sum_{j \in \JJJ} p_j \beta^*_j) w_{\max}}{\left(\Tilde{\lambda}^{*}\right)^2} \leq \frac{T}{\gamma} \nonumber \\
\Rightarrow &\sqrt{\frac{4(\sum_{j \in \JJJ} p_j \beta^*_j) w_{\max} \log(1/\eta)}{\left(\Tilde{\lambda}^{*}\right)^2 t^{(r-1)}}} \leq \sqrt{\frac{4T \log(2|\III|/\eta)}{t^{(r-1)} \gamma}} = \epsilon^{(r-1)}_{x}. \nonumber
\end{alignat}
\end{subequations}
Thus with a probability at least $1-\eta$,
\begin{subequations}
\begin{alignat}{2}
&\sum_{j \in \JJJ} \hat{p}^{(r)}_{j} \beta^*_j + \sum_{i \in \III} c_i \alpha_i \nonumber \\
\leq &\sum_{j \in \JJJ} p_j \beta^*_j + \sum_{i \in \III} c_i \alpha_i + \sqrt{\frac{4(\sum_{j \in \JJJ} p_j \beta^*_j) w_{\max} \log(1/\eta)}{t^{(r-1)}}} \nonumber \\
\leq &\Tilde{\lambda}^* (1+\epsilon^{(r-1)}_{x}). \nonumber
\end{alignat}
\end{subequations}
If $\sqrt{\frac{4 w_{\max} \log(1/\eta)}{t_{r-1}(\sum_{j \in \JJJ} p_j \beta^*_j)}} \geq 2e-1$, then the above procedure doesn't hold. Notice that all we need is $\sum_{j \in \JJJ} \hat{p}^{(r)}_{j} \beta^*_j - \sum_{j \in \JJJ} p_j \beta^*_j \leq \Tilde{\lambda}^* \epsilon^{(r-1)}_{x}$ with high probability. Hence define $\epsilon^{(r-1)}_{y}=\frac{\Tilde{\lambda}^* \epsilon^{(r-1)}_{x}}{\sum_{j \in \JJJ} p_j \beta^*_j}$. Supposing $\epsilon^{(r)}_{y}>2e-1$ (which will be proved later), then using the Chernoff bound, we have:
\begin{subequations}
\begin{alignat}{2}
&\Pr\left(\sum_{j \in \JJJ} \hat{p}^{(r)}_{j} \beta^*_j - \sum_{j \in \JJJ} p_j \beta^*_j \geq \Tilde{\lambda}^* \epsilon^{(r-1)}_{x}\right) \nonumber \\
= &\Pr\left(\sum_{j \in \JJJ} \hat{p}^{(r)}_{j} \beta^*_j \geq \sum_{j \in \JJJ} p_j \beta^*_j (1+\epsilon^{(r-1)}_{y})\right) \nonumber \\
\leq &2^{-\frac{(1+\epsilon^{(r-1)}_{y}) t^{(r-1)} \sum_{j \in \JJJ} p_j \beta^*_j}{w_{\max}}} \nonumber \\
= &2^{-\frac{t^{(r-1)} \sum_{j \in \JJJ} p_j \beta^*_j + t^{(r-1)} \Tilde{\lambda}^* \epsilon^{(r-1)}_{x}}{w_{\max}}} \nonumber \\
\leq &2^{-\frac{t^{(r-1)} \Tilde{\lambda}^* \epsilon^{(r-1)}_{x}}{w_{\max}}} \nonumber \\
\leq &2^{-\frac{t^{(r-1)} \gamma \epsilon^{(r-1)}_{x}}{T}} \nonumber \\
\leq &2^{-\epsilon^2 \gamma} \nonumber \\
\leq &\eta. \nonumber
\end{alignat}
\end{subequations}
The third inequality follows from $\gamma \leq \frac{T \Tilde{\lambda}^*}{w_{\max}}$. The fourth inequality holds by the definitions of $t^{(r-1)}$ and the fact that $\epsilon^{(r)}_{x} \geq \epsilon$ for all $r$. The last inequality holds if $\gamma \geq \frac{\log \frac{1}{\eta}}{\epsilon^2}$. This is indeed the case since $\gamma = \Omega\left(\frac{\log\frac{|\III|T}{\epsilon}}{\epsilon^2}\right)$ and we let $\eta=O\left(\frac{\epsilon}{l}\right)$ (which will be demonstrated later in the proof of Lemma 5). Hence the only thing left to prove is that given $\sqrt{\frac{4 w_{\max} \log(1/\eta)}{t^{(r-1)}(\sum_{j \in \JJJ} p_j \beta^*_j)}} \geq 2e-1$, we have $\epsilon^{(r)}_{y}>2e-1$ . This would be a direct result if $\sum_{j \in \JJJ} p_j \beta^*_j < \frac{\Tilde{\lambda}^* \epsilon}{2e-1}$. Next we show that given $\sum_{j \in \JJJ} p_j \beta^*_j \geq \frac{\Tilde{\lambda}^* \epsilon}{2e-1}$, $\sqrt{\frac{4 w_{\max} \log(1/\eta)}{t^{(r-1)}(\sum_{j \in \JJJ} p_j \beta^*_j)}} < 2e-1$. This completes the whole proof.
\begin{subequations}
\begin{alignat}{2}
&\sqrt{\frac{4 w_{\max} \log(1/\eta)}{t^{(r-1)}(\sum_{j \in \JJJ} p_j \beta^*_j)}} \nonumber \\
\leq & 2\sqrt{2e-1} \sqrt{\frac{w_{\max} \log(1/\eta)}{t^{(r-1)} \Tilde{\lambda}^* \epsilon}} \nonumber \\
\leq & 2\sqrt{2e-1} \sqrt{\frac{w_{\max} \log(1/\eta)}{\epsilon^2 T 2^{r-1} \Tilde{\lambda}^* \epsilon}} \nonumber \\
\leq & 2\sqrt{2e-1} \sqrt{\frac{\log(1/\eta)}{\epsilon^2 \gamma}} \nonumber \\
\leq &2\sqrt{2e-1} \nonumber \\
< &2e-1. \nonumber
\end{alignat}
\end{subequations}
The second inequality holds since $t^{(r-1)}=\epsilon T 2^{r-1}$. The third inequalities follow from $\gamma \leq \frac{T \Tilde{\lambda}^*}{w_{\max}}$. The fourth inequality holds since $\gamma \geq \frac{\log \frac{1}{\eta}}{\epsilon^2}$.
\end{proof}

\subsection{Proof of Lemma \ref{olem:3.4}}
Before proving Lemma \ref{olem:3.4}, we first demonstrate the validity of the following lemma.
\begin{lemma}\label{olem:A.2}
Suppose $\epsilon, \eta \in (0, 1)$. Define $\epsilon^{(r)}_{z}=\sqrt{\frac{2w_{\max} (1+\epsilon) \log\frac{2|\III|l}{\eta}}{t^{(r)} \lambda^{(r)}}}$ for $r \in \{0,1,\ldots,l-1\}$. In stage $r$, Algorithm $S$ achieves a reward of at least $$\sum^{t^{(r)}}_{t=1} \sum_{k \in \KKK} W^{(r)}_{i,j^{(r)}(t),k}(t) \Tilde{X}^{(r)}_{j^{(r)}(t),k}(t) \geq t^{(r)} \lambda^{(r)} (1-\epsilon^{(r)}_{z})$$ for every $i\in \III_r$ with probability at least $1- \left(1+\epsilon\right)^{\delta} (\frac{t^{(r)}}{2T} \epsilon + \frac{\eta}{2l})$. 
\end{lemma}

\begin{proof}[Lemma \ref{olem:A.2}]
To aid analysis, we leave $\bar{d}(\delta)$ customers unserved in the end of each stage, so the next stage starts nearly empty. This causes a total reward loss of $\bar{d}(\delta) w_{\max} \log \frac{1}{\epsilon}$. Suppose we run Algorithm $S$ over the entire planning horizon. Then at any time $t$ of stage $r$, the number of resource $i$ occupied by customers from previous stages can defined as $\hat{Y}^{(r)}_{it} = \sum^{r-1}_{h=0} \sum^{t^{(h)}}_{\tau=1} \sum_{k \in \KKK} \mathbf{1}(D^{(h)}_{i}(\tau) \geq t^{(r)}+t-(t^{(h)}+\tau)+1) A^{(h)}_{i,j^{(h)}(\tau),k}(\tau) \Tilde{X}^{(h)}_{j^{(h)}(\tau),k}(\tau)$. Since $\sum^{\infty}_{t = \bar{d}(\delta) +1} \Pr(D_i \geq t) \leq \delta$, it is evident that $\mathbb{E}[\hat{Y}^{(r)}_{it}] \leq a_{\max} \delta$. Then the probability of violating the resource constraint at the $t$-th time of stage $r$ is bounded as:
\begin{subequations}
\begin{alignat}{2}
\Pr(\hat{Y}^{(r)}_{it} + \sum^t_{\tau=1} \Tilde{Y}^{(r)}_{i \tau t} \geq \frac{c_i}{1+\epsilon}(1+\epsilon)) = & \Pr(\frac{\gamma}{c_i} (\hat{Y}^{(r)}_{it} + \sum^t_{\tau=1} \Tilde{Y}^{(r)}_{i \tau t}) \geq \frac{\gamma}{1+\epsilon}(1+\epsilon)) \nonumber \\
= & \Pr((1 + \epsilon)^{\frac{\gamma}{c_i} (\hat{Y}^{(r)}_{it} + \sum^t_{\tau=1} \Tilde{Y}^{(r)}_{i \tau t})} \geq (1 + \epsilon)^{\gamma}) \nonumber \\
\leq & \mathbb{E}[(1 + \epsilon)^{\frac{\gamma}{c_i} (\hat{Y}^{(r)}_{it} + \sum^t_{\tau=1} \Tilde{Y}^{(r)}_{i \tau t})}] / (1 + \epsilon)^{\gamma} \nonumber \\
\leq & \mathbb{E}[(1 + \epsilon)^{\frac{\gamma a_{\max} \delta}{c_i}} (1 + \epsilon)^{\frac{\gamma}{c_i} \sum^t_{\tau=1} \Tilde{Y}^{(r)}_{i \tau t}}] / (1 + \epsilon)^{\gamma} \nonumber \\
\leq & (1 + \epsilon)^{\delta - \gamma} \mathbb{E}[\prod^t_{\tau=1} (1 + \epsilon)^{\frac{\gamma}{c_i} \Tilde{Y}^{(r)}_{i \tau t}}] \nonumber \\
\leq & (1 + \epsilon)^{\delta - \gamma} \mathbb{E}[\prod^t_{\tau=1} (1 + \epsilon \frac{\gamma}{c_i} \Tilde{Y}^{(r)}_{i \tau t})] \nonumber \\
\leq & (1 + \epsilon)^{\delta - \gamma} \prod^t_{\tau=1} (1 + \epsilon \frac{\gamma}{c_i} \mathbb{E}[\Tilde{Y}^{(r)}_{i \tau t}]) \nonumber \\
\leq & (1 + \epsilon)^{\delta - \gamma} \prod^t_{\tau=1} \exp(\epsilon \frac{\gamma}{c_i} \mathbb{E}[\Tilde{Y}^{(r)}_{i \tau t}]) \nonumber \\
\leq & (1 + \epsilon)^{\delta} \exp(\frac{\epsilon \gamma}{1 + \epsilon}) / (1 + \epsilon)^{\gamma} \nonumber \\
\leq & (1 + \epsilon)^{\delta} \exp(\frac{-\epsilon^2 \gamma}{4(1 + \epsilon)}) \nonumber \\
\leq & \frac{\epsilon (1 + \epsilon)^{\delta}}{2|\III|T}. \nonumber
\end{alignat}
\end{subequations}
The first inequality follows from Markov's inequality. The second inequality holds since $\mathbb{E}[\hat{Y}^{(r)}_{it}] \leq a_{\max} \delta$. The third in equality holds since $\gamma \leq \frac{a_{\max}}{c_i} ~\forall i \in \III_c$. The fourth inequality stands since $\frac{\gamma}{c_i} \Tilde{Y}^{(r)}_{i \tau t} \leq 1$ for all $i \in \III_c$ and $1 \leq \tau \leq t \leq T$, and function $(1+\epsilon)^x \leq 1+\epsilon x$ for $x \in [0,1]$. The fifth inequality holds because for each time period $\tau$, $\Tilde{Y}^{(r)}_{i \tau t}$ are independent from each other. The sixth inequality holds by the fact that function $1+x \leq e^x$. The seven inequality follows from $\mathbb{E}\left[\Tilde{Y}^{(r)}_{i \tau t}\right] \leq \frac{c_i}{1+\epsilon}$. The eighth inequality stands for all $\epsilon \in [0,1]$, and the last inequality follows by assuming $\gamma \geq \frac{4(1+\epsilon)}{\epsilon^2} \log \frac{2|\III| T }{\epsilon}$. Taking a union bound over all $i \in \III_c$ and $t = 1, \ldots, t^{(r)}$, the probability of satisfying all resource constraints for $\text{(LP-RSS)}^{(r)}$ is at least $$1-\sum^{t^{(r)}}_{\tau=1} \sum_{i \in \III_c} \frac{\epsilon (1 + \epsilon)^{\delta}}{2|\III|T} = 1 - \frac{\epsilon (1 + \epsilon)^{\delta}}{2|\III|T} |\III_c| t^{(r)} \geq 1-\frac{t^{(r)}}{2T} \epsilon (1 + \epsilon)^{\delta}.$$ For the reward constraints, notice $\epsilon^{(r)}_{z}=\sqrt{\frac{2w_{\max} (1+\epsilon) \log\frac{2|\III|l}{\eta}}{t^{(r)} \lambda^{(r)}}}$, we have:
\begin{subequations}
\begin{alignat}{2}
&\Pr\left(\sum^{t^{(r)}}_{t=1} \sum_{k \in \KKK} W^{(r)}_{i,j^{(r)}(t),k}\left(t\right) \Tilde{X}^{(r)}_{j^{(r)}(t),k}\left(t\right) \leq t^{(r)} \lambda^{(r)} \left(1-\epsilon^{(r)}_{z}\right)\right) \nonumber \\
=& \Pr\left(\sum^{t^{(r)}}_{t=1} \sum_{k \in \KKK} \frac{W^{(r)}_{i,j^{(r)}(t),k}\left(t\right)}{w_{\max}} \Tilde{X}^{(r)}_{j^{(r)}(t),k}\left(t\right) \leq \frac{t^{(r)} \lambda^{(r)} \left(1-\epsilon^{(r)}_{z}\right)}{w_{\max}}\right) \nonumber \\
\leq & \exp\left(\frac{-\left(\epsilon^{(r)}_{z}\right)^2 t^{(r)} \lambda^{(r)}}{2w_{\max} \left(1+\epsilon\right)}\right) \nonumber \\
\leq & \frac{\eta}{2|\III|l} \nonumber.
\end{alignat}
\end{subequations}
The first inequality follows from applying the multiplicative Chernoff bounds on $\sum_{k \in \KKK} \frac{W^{(r)}_{i,j(t),k}(t)}{w_{\max}} \Tilde{X}^{(r)}_{j(t),k}(t)$. It is evident that these summing terms are mutually independent and within range $[0,1]$, and meanwhile $\mathbb{E}\left[\sum^{t^{(r)}}_{t=1} \sum_{k \in \KKK} \frac{W^{(r)}_{i,j(t),k}(t)}{w_{\max}} \Tilde{X}^{(r)}_{j^{(r)}(t),k}(t)\right] \geq \frac{t^{(r)} \Tilde{\lambda}^*}{w_{\max} (1+\epsilon)} \geq \frac{t^{(r)} \lambda^{(r)}}{w_{\max} (1+\epsilon)}$. The last inequality holds by the definition of $\epsilon^{(r)}_{z}$. By a union bound over resource and reward constraints, we have the result of Lemma \ref{olem:A.2}.
\end{proof}

\begin{proof}[Lemma \ref{olem:3.4}]
Suppose $\Tilde{\lambda}^*(1-3\epsilon^{(r-1)}_{x}) \leq \lambda^{(r)} \leq \Tilde{\lambda}^*$ holds. Using $\lambda^{(r)}$ as an approximation of $\lambda$, over the $l$ stages, we achieve a total reward of 
\begin{subequations}
\begin{alignat}{2}
& \sum^{l-1}_{r=0} t^{(r)} \lambda^{(r)} \left(1-\epsilon^{(r)}_{z}\right) \nonumber \\
\geq &\sum^{l-1}_{r=0} T \Tilde{\lambda}^* \frac{t^{(r)}}{T} \left(1 - 3\epsilon^{(r-1)}_{x}\right) \left(1 - \epsilon^{(r)}_{z}\right) \nonumber \\
\geq & T \Tilde{\lambda}^* \sum^{l-1}_{r=0}  \frac{t^{(r)}}{T} \left(1 - 3\epsilon^{(r-1)}_{x} - \epsilon^{(r)}_{z}\right) \nonumber \\
\geq & T \Tilde{\lambda}^* \left(1-O\left(\sum^{l-1}_{r=0} \epsilon \sqrt{\frac{t^{(r)}}{T}}\right)\right) \nonumber \\
\geq & T \Tilde{\lambda}^* \left(1-O\left(\epsilon \sqrt{\epsilon} \sum^{l-1}_{r=0} \sqrt{2}^r\right)\right) \nonumber \\
\geq & T \Tilde{\lambda}^* \left(1-O\left(\epsilon \sqrt{\epsilon} \sum^{\log_{\sqrt{2}}\left(\frac{1}{\sqrt{\epsilon}}\right)}_{r=0} \sqrt{2}^r\right)\right) \nonumber \\
\geq & T \Tilde{\lambda}^* \left(1-O\left(\epsilon\right)\right). \nonumber
\end{alignat}
\end{subequations}
The first inequality holds since $\Tilde{\lambda}^*(1-3\epsilon^{(r-1)}_{x}) \leq \lambda^{(r)} \leq \Tilde{\lambda}^*$. The third inequality holds by definitions of $\epsilon^{(r)}_{x} = \sqrt{\frac{4 T \log \frac{2 |\III|}{\eta}}{t^{(r)} \gamma}}$ and $\epsilon^{(r)}_{z}=\sqrt{\frac{2w_{\max} (1+\epsilon) \log \frac{2|\III|l}{\eta}}{t^{(r)} \lambda^{(r)}}}$, and the fact that $\gamma \leq \frac{T \lambda^{(r)}}{w_{\max} (1-3\epsilon^{(r-1)}_{x})}$. The fourth and fifth inequalities stems from $t^{(r)}=\epsilon T 2^r$ and $\epsilon 2^l=1$ respectively.

Notice that with probability of at most $2\eta$, $\Tilde{\lambda}^*(1-3\epsilon^{(r-1)}_{x}) \leq \lambda^{(r)} \leq \Tilde{\lambda}^*$ fails. Therefore, Algorithm $S$ gives us a reward of at least $T \Tilde{\lambda}^* (1 - O(\epsilon))$ with probability of at least $$1 - \sum^{l-1}_{r=0} \frac{t^{(r-1)}}{2T} \epsilon (1 + \epsilon)^{\delta} - \sum^{l-1}_{r=0}  \frac{\eta}{2l} - 2\eta= 1-\frac{\epsilon(1 + \epsilon)^{\delta}}{2} - \frac{5\eta}{2} \leq 1 - \epsilon (1 + \epsilon)^{\delta}.$$
\end{proof}
 
\subsection{Proof of Lemma \ref{olem:3.5}}
\begin{proof}
To aid the analysis, we leave $\bar{d}(\delta)$ customer unserved in the end of each stage, so the next stage starts nearly empty. This causes a total reward loss of at most $\bar{d}(\delta) w_{\max} l = \bar{d}(\delta) w_{\max} \log (1/\epsilon)$. Define the number of resource $i$ occupied by customers from previous stages as $\hat{Y}^{(r)}_{it}$ in time $t$ of stage $r$. It is evident that $\mathbb{E}[\hat{Y}^{(r)}_{it}] \leq a_{\max} \delta$. Then at any time step $t \in \{s+1,\ldots,t^{(r)}\}$, by the Markov inequality, the probability of violating resource $i$ constraint is:
\begin{align}
&\Pr\left(\hat{Y}^{(r)}_{it}+ \sum^{s+1}_{\tau=1} Y^{(r)}_{i \tau t}+ \sum^{t}_{\tau=s+2} \Tilde{Y}^{(r)}_{i \tau t} \geq \frac{c_i}{1+\epsilon}\left(1+\epsilon\right)\right) \nonumber \\
= & \Pr\left(\frac{\gamma}{c_i} \left(\hat{Y}^{(r)}_{it}+ \sum^{s+1}_{\tau=1} Y^{(r)}_{i \tau t}+ \sum^{t}_{\tau=s+2} \Tilde{Y}^{(r)}_{i \tau t}\right) \geq \gamma \right) \nonumber \\
\leq & \mathbb{E}\left[(1+\epsilon)^{\frac{\gamma}{c_i} \left(\hat{Y}^{(r)}_{it}+ \sum^{s+1}_{\tau=1} Y^{(r)}_{i \tau t}+ \sum^{t}_{\tau=s+2} \Tilde{Y}^{(r)}_{i \tau t}\right)} \right] / (1+\epsilon)^{\gamma} \nonumber \\
\leq & \mathbb{E}\left[(1+\epsilon)^{\frac{\gamma a_{\max} \delta}{c_i}} \left(1+\epsilon\right)^{\frac{\gamma}{c_i} \sum^{s}_{\tau=1} Y^{(r)}_{i \tau t}} \left(1+\epsilon \frac{\gamma}{c_i} Y^{(r)}_{i,s+1,t}\right)\left(1+\epsilon\right)^{\sum^{t}_{\tau=s+2} \frac{\gamma}{c_i}  \mathbb{E}\left[\Tilde{Y}^{(r)}_{i \tau t}\right]}\right] / (1+\epsilon)^{\gamma} \nonumber \\
\leq &  (1+\epsilon)^{\delta - \gamma} \mathbb{E}\left[\left(1+\epsilon\right)^{\frac{\gamma}{c_i} \sum^{s}_{\tau=1} Y^{(r)}_{i \tau t}} \left(1+\epsilon \frac{\gamma}{c_i} \Tilde{Y}^{(r)}_{i,s+1,t}\right) \prod^{t}_{\tau=s+2} \left(1+\epsilon \frac{\gamma}{c_i}  \mathbb{E}\left[\Tilde{Y}^{(r)}_{i \tau t}\right]\right)\right].  \label{oeq:res_vio}
\end{align}
While it is desirable to upper bound the probabilities of violating any resource constraints without the term $\mathbb{E}[\Tilde{Y}^{(r)}_{i \tau t}]$, the values of $\mathbb{E}[\Tilde{Y}^{(r)}_{i \tau t}]$ are not known in the online setting. Nevertheless, we use the fact that $\sum_{k \in \KKK} \sum_{j \in \JJJ} p_j a_{ijk} d_i x^*_{jk} \leq c_i$ for the ``steady-state'' optimal solution, and upper bound $\mathbb{E}[\Tilde{Y}^{(r)}_{i \tau t}]$ as follows:
\begin{align}
\mathbb{E}\left[\Tilde{Y}^{(r)}_{i \tau t}\right] & = \sum_{k \in \KKK} \sum_{j \in \JJJ} p_j a_{ijk} \Pr(D_i \geq t - \tau +1) \frac{x^*_{jk}}{1+\epsilon} \nonumber \\
& \leq \frac{c_i}{d_i(1+\epsilon)} \Pr(D_i \geq t - \tau +1). \label{oeq:tildeY}
\end{align}
Plugging (\ref{oeq:tildeY}) in (\ref{oeq:res_vio}), in time step $s$, we aim to minimize the following term which, by the union over $t \in \{s+1, \ldots, t^{(r)}\}$ and $i \in \III_c$, upper bounds the probability of violating at least one future resource constraint in stage $r$:
\begin{align}
& \sum^{t^{(r)}}_{t=s+1} \sum_{i \in \III_c} \mathbb{E}\left[\left(1+\epsilon\right)^{\frac{\gamma}{c_i} \sum^{s}_{\tau=1} Y^{(r)}_{i \tau t}} \left(1+\epsilon \frac{\gamma}{c_i} Y^{(r)}_{i,s+1,t}\right) \prod^{t}_{\tau=s+2} \left(1+\epsilon \frac{\gamma \Pr\left(D_i \geq t-\tau+1\right)}{d_i \left(1+\epsilon\right)}\right)\right] (1+\epsilon)^{\delta - \gamma}. \label{oeq:unionY}
\end{align}

Similarly, since $\mathbb{E}[\Tilde{Z}_{it}]=\sum_{k \in \KKK} \sum_{j \in \JJJ} p_j w_{ijk} \frac{x^*_{jk}}{1+\epsilon} \geq \frac{\Tilde{\lambda}^*}{1+\epsilon} \geq \frac{\lambda^{(r)}}{1+\epsilon}$, by the Markov inequality, we upper bound the reward constraint violation probabilities in stage $r$ as:
\begin{align}
& \sum_{i \in \III_r} \Pr\left(\sum^{s+1}_{\tau=1} Z^{(r)}_{i \tau}+ \sum^{t^{(r)}}_{\tau=s+2} \Tilde{Z}^{(r)}_{i \tau} \leq t^{(r)} \lambda^{(r)} \left(1-\epsilon^{(r)}_{z}\right)\right) \nonumber \\
& \leq \sum_{i \in \III_r} \mathbb{E}\left[\left(1-\epsilon^{(r)}_{z}\right)^{\frac{1}{w_{\max}} \sum^{s}_{\tau=1} Z^{(r)}_{i \tau}} \left(1-\epsilon^{(r)}_{z} \frac{1}{w_{\max}} Z^{(r)}_{i,s+1}\right) \prod^{t^{(r)}}_{\tau=s+2} \left(1-\epsilon^{(r)}_{z} \frac{\lambda^{(r)}}{w_{\max} \left(1+\epsilon\right)}\right)\right] / \left(1-\epsilon^{(r)}_{z}\right)^{\frac{\left(1-\epsilon^{(r)}_{z}\right) t^{(r)} \lambda^{(r)}}{w_{\max}}}. \label{oeq:unionZ}
\end{align}
Define the sum of terms (\ref{oeq:unionY}) and (\ref{oeq:unionZ}) as $\mathcal{F}^{(r)}(A^{s+1} S^{t^{(r)}-s-1})$. In time step $s+1 \in \{1,\ldots,t^{(r)}\}$ of stage $r$, the DM aims to take the action $k^{(r)}(s+1)$ that minimizes $\mathcal{F}^{(r)}(A^{s+1} S^{t^{(r)}-s-1})$, which has exactly the form of (\ref{oalg:k}). We further show that $k^{(r)}(s+1)$ in fact implies that $\mathcal{F}^{(r)}(A^{s+1} S^{t^{(r)}-s-1}) \leq \mathcal{F}^{(r)}(A^s S^{t^{(r)}-s})$ as follows, and inductively we have $\mathcal{F}^{(r)}(A^T) \leq \mathcal{F}^{(r)}(S^T)$.

Suppose we run Algorithm $S$ in the first $s$ time periods of stage $r$, then at time period $s+1$, Algorithm $A$ obviously minimizes $\mathcal{F}_r\left(A^{s+1} S^{t^{(r)}-s-1}\right)$. Hence, if we replace the action Algorithm $A$ at time $s+1$ by the choice of Algorithm $S$, the probability of constraint violation is no smaller:
\begin{subequations}
\begin{alignat}{2}
&\mathcal{F}_r\left(A^{s+1} S^{t^{(r)}-s-1}\right) \nonumber \\
\leq & \frac{\sum^{t^{(r)}}_{t=s+1} \sum_{i \in \III_c} \mathbb{E}\left[\left(1+\epsilon\right)^{\frac{\gamma}{c_i} \sum^{s}_{\tau=1} Y^{(r)}_{i \tau t}} \left(1+\epsilon \frac{\gamma}{c_i} Y^{(r)}_{i,s+1,t}\right) \prod^{t}_{\tau=s+2} \left(1+\epsilon \frac{\gamma \Pr\left(D_i \geq t-\tau+1\right)}{d_i \left(1+\epsilon\right)}\right)\right]}{\left(1+\epsilon\right)^{\gamma-\delta}}  \nonumber \\
& + \frac{\sum_{i \in \III_r} \mathbb{E}\left[\left(1-\epsilon^{(r)}_{z}\right)^{\frac{1}{w_{\max}} \sum^{s}_{\tau=1} Z^{(r)}_{i \tau}} \left(1-\epsilon^{(r)}_{z} \frac{1}{w_{\max}} Z^{(r)}_{i,s+1}\right) \prod^{t^{(r)}}_{\tau=s+2} \left(1-\epsilon^{(r)}_{z} \frac{\lambda^{(r)}}{w_{\max} \left(1+\epsilon\right)}\right)\right]}{\left(1-\epsilon^{(r)}_{z}\right)^{\frac{\left(1-\epsilon^{(r)}_{z}\right) t^{(r)} \lambda^{(r)}}{w_{\max}}}} \nonumber \\
\leq & \frac{\sum^{t^{(r)}}_{t=s+1} \sum_{i \in \III_c} \mathbb{E}\left[\left(1+\epsilon\right)^{\frac{\gamma}{c_i} \sum^{s}_{\tau=1} Y^{(r)}_{i \tau t}} \left(1+\epsilon \frac{\gamma}{c_i} \Tilde{Y}^{(r)}_{i,s+1,t}\right) \prod^{t}_{\tau=s+2} \left(1+\epsilon \frac{\gamma \Pr\left(D_i \geq t-\tau+1\right)}{d_i \left(1+\epsilon\right)}\right)\right]}{\left(1+\epsilon\right)^{\gamma-\delta}}  \nonumber \\
& + \frac{\sum_{i \in \III_r} \mathbb{E}\left[\left(1-\epsilon^{(r)}_{z}\right)^{\frac{1}{w_{\max}} \sum^{s}_{\tau=1} Z^{(r)}_{i \tau}} \left(1-\epsilon^{(r)}_{z} \frac{1}{w_{\max}} \Tilde{Z}^{(r)}_{i,s+1}\right) \prod^{t^{(r)}}_{\tau=s+2} \left(1-\epsilon^{(r)}_{z} \frac{\lambda^{(r)}}{w_{\max} \left(1+\epsilon\right)}\right)\right]}{\left(1-\epsilon^{(r)}_{z}\right)^{\frac{\left(1-\epsilon^{(r)}_{z}\right) t^{(r)} \lambda^{(r)}}{w_{\max}}}} \nonumber \\
\leq & \frac{\sum^{t^{(r)}}_{t=s+1} \sum_{i \in \III_c} \mathbb{E}\left[\left(1+\epsilon\right)^{\frac{\gamma}{c_i} \sum^{s}_{\tau=1} Y^{(r)}_{i \tau t}} \prod^{t}_{\tau=s+1} \left(1+\epsilon \frac{\gamma \Pr\left(D_i \geq t-\tau+1\right)}{d_i \left(1+\epsilon\right)}\right)\right]}{\left(1+\epsilon\right)^{\gamma-\delta}}  \nonumber \\
& + \frac{\sum_{i \in \III_r} \mathbb{E}\left[\left(1-\epsilon^{(r)}_{z}\right)^{\frac{1}{w_{\max}} \sum^{s}_{\tau=1} Z^{(r)}_{i \tau}} \sum^{t^{(r)}}_{\tau=s+2} \left(1-\epsilon^{(r)}_{z} \frac{\lambda^{(r)}}{w_{\max} \left(1+\epsilon\right)}\right)\right]}{\left(1-\epsilon^{(r)}_{z}\right)^{\frac{\left(1-\epsilon^{(r)}_{z}\right) t^{(r)} \lambda^{(r)}}{w_{\max}}}} \nonumber \\
& \leq \mathcal{F}_r\left(A^s S^{t^{(r)}-s}\right). \nonumber
\end{alignat}
\end{subequations}
By induction, we have $\mathcal{F}_r\left(A^{t^{(r)}}\right) \leq \mathcal{F}_r\left(S^{t^{(r)}}\right)$. Since Algorithm $A$ has no worse performance than Algorithm $S$ in each stage, following Lemma \ref{olem:3.4}, the total probability of constraint violation over the entire planning horizon $$\sum^{l-1}_{r=0} \mathcal{F}_r\left(A^{t^{(r)}}\right) \leq \sum^{l-1}_{r=0} \mathcal{F}_r\left(S^{t^{(r)}}\right) \leq \epsilon (1 + \epsilon)^{\delta}.$$ Therefore, Algorithm $A$ achieves a reward of at least $T \Tilde{\lambda}^* (1 - O(\epsilon))$ for every $i \in \III_r$ with probability at least $1 - \epsilon (1 + \epsilon)^{\delta}$.
\end{proof}

\subsection{More on the numerical experiments}

It is worth mentioning that for the assortment planning applications, the size of the action set $\KKK$ scales exponentially with the number of products, and therefore equation (1) is computationally intractable under a general choice model. However, we don't need to enumerate all possible assortments in certain cases, say under MNL models. For simultaneously maximizing the revenue of each resource, we set $\III_r=\III_c$, $A_{ijk} \sim \text{Bernoulli} (q_{ijk})$ and $W_{ijk} = r_i A_{ijk}$. In this case, equation (1) can be written as:
\begin{equation}
k^{(r)}(s+1)=\arg \min \limits_{k \in \KKK } \left(\sum_{i \in \III_c} \left(\sum^{t^{(r)}}_{t=s+1} \Pr( D_i \geq t-s) \phi^{(r)}_{i,s+1,t} + r_i \psi^{(r)}_{i,s+1}\right) q_{i,j^{(r)}(s+1),k} \right). \tag{4}
\end{equation}
Under a MNL model, $q_{ijk} = \frac{\exp(\boldsymbol{b}_{ij}^\top \boldsymbol{f}_i)}{1+\sum_{\ell \in k} \exp(\boldsymbol{b}_{\ell j}^\top \boldsymbol{f}_\ell)}$ if $i\in k$, and $q_{ijk} = 0$ if $i\not \in k$. It can be shown that for a type $j$ customer, (4) can be solved efficiently by solving the following LP (\cite{davis2013assortment}):
\begin{align}
\min  &~\sum_{i \in \III_c}  \left(\sum^T_{t=s+1} \Pr( D_i \geq t-s) \phi^{(r)}_{i,s+1,t} +  r_i \psi^{(r)}_{i,s+1}\right) z_i & \nonumber\\
\text{s.t.}  & \sum_{i \in \III_c} z_i + z_0 = 1   \nonumber\\
& \sum_{i \in \III_c} \frac{z_i}{\boldsymbol{b}_{ij}^\top \boldsymbol{f}_i} \leq n z_0 \nonumber\\
& 0 \leq \frac{z_i}{\boldsymbol{b}_{ij}^\top \boldsymbol{f}_i} \leq z_0 \quad \forall i\in \III_c \nonumber
\end{align}
where the decision variables are $\{z_i: i \in \III_c \cup \{0\}\}$. 

In our numerical tests, when the number of customer types is large, the computation time is large despite the usage of column generation method. We can use the sample average approximation (SAA) to reduce the number of customer types. Specifically, we can sample 100 times from the 1000 customer types and approximate the value of $\Tilde{\lambda}^*$ on the 100 samples. This technique is adopted in Algorithm 3, where we have shown that with large enough sample sizes, the SAA works well. Hence essentially we are performing SAA on SAA when the size of $\JJJ$ is large. Table 1 shows the results and computational time comparisons between the original algorithm $A$ and algorithm $A$ with SAA.

\begin{table}[h]
\centering
\resizebox{\textwidth}{!}{%
\begin{tabular}{|c|c|c|cc|cc|cc|}
\hline
\multirow{2}{*}{T} &
  \multirow{2}{*}{$(\gamma,\epsilon)$} &
  \multirow{2}{*}{UB} &
  \multicolumn{2}{c|}{Total Revenue} &
  \multicolumn{2}{c|}{$\%$ Gap from UB} &
  \multicolumn{2}{c|}{Computational Time} \\
 &
   &
   &
  $A$ &
  SAA &
  $A$ &
  SAA &
  $A$ &
  SAA \\ \hline
1000 & (300, 0.3)    & 644.90  & 331.03  & 305.90  & 48.67\% & 52.57\% & 134.18  & 55.19  \\
2000 & (600, 0.22)   & 1289.80 & 887.28  & 838.60  & 31.21\% & 34.98\% & 260.01  & 103.77 \\
3000 & (900, 0.185)  & 1934.70 & 1419.54 & 1357.70 & 26.63\% & 29.82\% & 390.10  & 143.35 \\
4000 & (1200, 0.162) & 2579.60 & 1966.83 & 1861.60 & 23.75\% & 27.83\% & 489.99  & 188.43 \\
5000 & (1500, 0.148) & 3224.50 & 2522.00 & 2398.70 & 21.79\% & 25.61\% & 576.65  & 246.04 \\
6000 & (1800, 0.136) & 3869.40 & 3086.28 & 2934.40 & 20.24\% & 24.16\% & 641.84  & 297.80 \\
7000 & (2100, 0.127) & 4514.30 & 3647.95 & 3477.90 & 19.19\% & 22.96\% & 732.04  & 352.72 \\
8000 & (2400, 0.12)  & 5159.20 & 4217.17 & 4014.40 & 18.26\% & 22.19\% & 1114.00 & 687.89 \\ \hline
\end{tabular}%
}
\caption{Using SAA to reduce computational time.}
\label{tab:my-table2}
\end{table}

\end{document}